\documentclass[11pt]{article}

\usepackage{amsmath,amsfonts,amsthm,euscript,amssymb,psfrag,subfigure,capt-of}
\usepackage[dvips]{graphicx}
\usepackage{hyperref}
\usepackage{pdflscape}

\newtheorem{theorem}{Theorem}

\newtheorem{lemma}[theorem]{Lemma}
\newtheorem{prop}[theorem]{Proposition}
\newtheorem{corollary}[theorem]{Corollary}

\newtheorem{conjecture}[theorem]{Conjecture}

\numberwithin{equation}{section}
\numberwithin{theorem}{section}

\newcommand{\A}{\mathbb{A}}

\newcommand{\E}{\EuScript{E}}
\newcommand{\F}{\mathbb{F}}

\newcommand{\Q}{\mathbb{Q}}

\newcommand{\Z}{\mathbb{Z}}

\newcommand{\FF}{\EuScript{F}}
\newcommand{\QQ}{\mathcal{Q}}

\renewcommand{\O}{\EuScript{O}}
\renewcommand{\P}{\mathbb{P}}

\newcommand{\ord}{\mathrm{ord}}

\newcommand{\Fbar}{\overline{\F}}

\newcommand{\loccit}{{\it loc.~cit.}}

\newcommand{\vareps}{\varepsilon}

\newcommand{\cpp}{{\small C++}}
\newcommand{\sage}{{\small Sage}}
\newcommand{\ellff}{{\small ELLFF}}
\newcommand{\bigellff}{{ ELLFF}}

\begin{document}

\title{Experimental Data for Goldfeld's Conjecture over Function Fields}
\author{Salman Baig
\!\thanks{Department of Mathematics, University of Washington,
\ {\tt salmanhb@math.washington.edu}}
\ and\ 
Chris Hall
\!\thanks{Department of Mathematics, University of Wyoming, Laramie,
\ {\tt chall14@uwyo.edu}}}
\maketitle

\begin{abstract}
\noindent This paper presents empirical evidence supporting Goldfeld's conjecture on the average analytic rank of a family of quadratic twists of a fixed elliptic curve in the function field setting. In particular, we consider representatives of the four classes of non-isogenous elliptic curves over $\F_q(t)$ with $(q,6)=1$ possessing two places of multiplicative reduction and one place of additive reduction. The case of $q=5$ provides the largest data set as well as the most convincing evidence that the average analytic rank converges to $1/2$, which we also show is a lower bound following an argument of Kowalski. The data was generated via explicit computation of the $L$-function of these elliptic curves, and we present the key results necessary to implement an algorithm to efficiently compute the $L$-function of non-isotrivial elliptic curves over $\F_q(t)$ by realizing such a curve as a quadratic twist of a pullback of a `versal' elliptic curve. We also provide a reference for our open-source library \ellff, which provides all the necessary functionality to compute such $L$-functions, and additional data on analytic rank distributions as they pertain to the density conjecture.
\end{abstract}

\section{Introduction}

Goldfeld's conjecture, in its original form, makes an assertion about a family of elliptic curves over a number field and some form of rank.  For example, if we fix an elliptic curve $E/\Q$ and consider the set of its quadratic twists ordered by (increasing) discriminant of the twisting fields, then the conjecture asserts that the average rank of the first $n$ curves tends to the limit $1/2$ as $n$ tends to infinity.  In this paper we will fix the rational function field $K=\F_q(t)$ as our base field and consider families of elliptic curves over $K$ for which we can calculate each family member's analytic rank.  Little theoretical progress has been made when we consider the average rank of an `increasing' sequence of curves, and the genesis of this paper lies in a computational project to generate a rich data set of $L$-functions for studying the (empirical) averages.

In the bulk of this paper we focus on some algorithms of increasing complexity for explicitly calculating the $L$-function $L(E/K,s)$ for an elliptic curve $E/K$; the increasing complexity allows for increased speed but at the cost of increasing the disk and memory space requirements.  In contrast to $L$-functions over number fields, these $L$-functions have the remarkable property that if we define a new variable $T$ by $T=q^{-s}$, then we can represent the $L$-function as a polynomial $L(E/K,T)\in 1+T\cdot\Z[T]$.  The analytic rank of $E/K$ is simply the order of vanishing at $s=1$ ($T=1/q$), hence the analytic Goldfeld's conjecture requires only a miniscule amount of the information present in each $L$-function.\footnote{The Birch--Swinnerton-Dyer conjecture asserts that the algebraic (Mordell-Weil) and analytic ranks are the same, so someone who likes looking for points may want to comb our database for curves of rank, preferably at least two, and try their hand at producing points.}  However, there are many other questions one can ask about how $L(E/K,T)$ varies with $E/K$, so we hope that our database and algorithms will prove useful for others.
 
For the calculation of a single $L$-function, one can use the theory outlined in the first two sections of section~\ref{sec::ec}.  We follow the usual approach for calculating the $L$-function by expressing it as an Euler product; the terms of the product are indexed by valuations in $K$, i.e.~by monic irreducibles in the polynomial ring $\F_q[t]$ and a point at infinity.  The Euler product is infinite, but there is a finite collection of Euler factors that completely determine $L(E/K,T)$.  To compute the $L$-function as efficiently as possible (in this approach) we want to minimize the size of this collection (e.g.~by using a functional equation) and the cost of computing a single Euler factor.

Those who have experience computing quadratic twists know that the cost of computing an Euler factor for a twist is much cheaper if one knows the Euler factor for the original curve, the work being reduced to computing a Legendre symbol.  A similar efficiency emerges when we consider pullbacks, i.e.~when we replace $K$ with a finite extension $L/K$: almost all of the Euler factors for $L(E/L,T)$ can be cheaply calculated from the Euler factors of $L(E/K,T)$.  In fact, for each $q$, there is an elliptic curve $E_0$ over $F=\F_q(j)$ such that every $E/K$ may be written as the combination of a pullback of $E_0$ to $K$ (via the embedding $F\to K$ induced by $j\mapsto j(E)$) and a twist (by quadratic $L/K$).  In particular, the Euler factors for $L(E/K,T)$ may be cheaply computed from the Euler factors of $L(E_0/F,T)$, and the second half of section~\ref{sec::ec} explains this in detail.

The upshot is that whenever we construct a new curve $E/K$ as a pullback or twist of another curve $E_0/F$, then the additional cost of computing $L(E/K,T)$ using a precomputed table of sufficiently many Euler factors for $L(E_0/F,T)$ is much cheaper than if we computed $L(E/K,T)$ from scratch.  We have written a library for calculating $L$-functions whose core routines use the methods we outline in section~\ref{sec::ec} and recently provided the wrapper routines for it to be used within \sage~ \cite{Sage}.  The library is called \ellff~(for elliptic $L$-functions over function fields), and we describe its basic structure in section~\ref{sec::ellff}.  We used our code to gather data to empirically study Goldfeld's conjecture in the function field setting, and we report a summary of this data and observations in section~\ref{sec::compute}.

\section{Theoretical Framework}\label{sec::ec}

\subsection{Elliptic Curves}

Fix a prime power $q$ relatively prime to 6.  Let $K=\F_q(t)$ be the function field of the curve $\P^1/\F_q$, $\O_K=\F_q[t]$ be the affine coordinate ring of $\P^1-\{\infty\}$, $\O_\infty=\F_q[u]$ be the affine coordinate ring of $\P^1-\{0\}$, and $u=1/t\in K$.  We identify the set of closed points $|\P^1|=\{\pi\}$ with the set formed by the closed point $\pi=\infty$ together with the monic irreducibles $\pi\in\F_q[t]$ of positive degree, and we write $\F_\pi$ for the residue field.  For $\pi=\infty$ we identify $\F_\pi$ with the quotient field $\F_q[u]/u$, and otherwise we identify $\F_\pi$ with the quotient field $\F_q[t]/\pi$.

Let $E/K$ be an elliptic curve.  Up to a change of coordinates, we may represent our elliptic curve as the projective plane curve given by the affine curve $y^2=x^3+ax+b$, where $a,b\in K$, together with the point at infinity.  The discriminant $\Delta=4a^3+27b^2$ and $j$-invariant $j=6912a^3/\Delta$ are rational functions in $t$, i.e.~elements of $K$, and $\Delta\neq 0$.

Up to a change of coordinates $(x,y)\mapsto(x/g^2,y/g^3)$ with $g\in K^\times$, we may assume $a,b\in\O_K$ and $\deg_t(\Delta)$ is minimal because $\O_K$ is a principal ideal domain, and we write $\Delta_\pi=\Delta$ for $\pi\neq\infty$.  There is a unique integer $e$ such that the substitution $(t,x,y)\mapsto (1/u,x/u^{2e},y/u^{3e})$ yields a model of $E$ over $\F_q[u]$ satisfying similar conditions and we write $\Delta_\infty$ for the discriminant of this model.  We glue the two models together over the annulus $\P^1-\{0,\infty\}$ to form the so-called minimal Weierstrass model $\E\to\P^1$; it is the identity component of the so-called N\'eron model of $E$.

For each $\pi$, we write $\E/\F_\pi$ for the fiber of $\E\to\P^1$ over $\pi$.  If $\pi\neq\infty$, then it is the projective plane curve given by `reducing modulo $\pi$' the model for $E$ over $\O_K$, while for $\pi=\infty$ we use the model for $E$ over $\F_q[u]$.  $\E/\F_\pi$ is a smooth curve if and only if the image of $\Delta_\pi$ in $\F_\pi$ is non-zero, otherwise it has a unique singular point.  We write $M$ for the finite subset of $\pi$ such that $\E/\F_\pi$ is singular with a node, $A$ for the finite subset such that $\E/\F_\pi$ is singular with a cusp, and $U$ for the open complement $\P^1-M-A$; they are the loci of multiplicative, additive, and good reduction respectively of $\E\to\P^1$.  We recall that if $\pi\neq\infty\in M\cup A$, then $\pi\in M$ if and only if the image of $b$ in $\F_\pi$ is non-zero, and a similar criterion holds if $\pi=\infty\in M\cup A$.

We decompose $M$ into the subset $M^+$ of $\pi$ for which the slopes of the two branches through the node of $\E/\F_\pi$ are rational over $\F_\pi$ and $M^-$ for the complement $M-M^+$; they are the loci of split and non-split reduction respectively.  The following lemma gives a criterion for deciding whether any $\pi\neq\infty\in M$ lies in $M^+$ or $M^-$, while for $\pi=\infty$ one can use the model for $E$ over $\F_q[u]$ to deduce a similar criterion.

\begin{lemma}
    If $\pi\neq\infty\in M$, then $\pi\in M^+$ if and only if the image of $6b$ in $\F_\pi$ is a square.
\end{lemma}

\begin{proof}
Over $\F_\pi$ our affine model specializes to $y^2=(x-c)^2(x+2c)$ for some $c\neq 0\in\Fbar_\pi$ and the node lies at $(x,y)=(c,0)$.  The substitution $y=s\cdot(x-c)$ and cancellation of $(x-c)^2$ leads to the curve $s^2=x+2c$, the blow up of our original curve at the node.  The slopes of the two branches are the $s$-coordinates of the points $(x,s)=(c,s)$ which lie on this curve, hence $s=\pm\sqrt{3c}$.  In particular, the slopes are rational over $\F_\pi$ if and only if $3c$ is a square in $\F_\pi$.  In terms of the singular model, we see that the image of $b$ in $\F_\pi$ is $2c^3$, hence $3c$ is a square in $\F_\pi$ if and only if $6b=3c\cdot(2c)^2$ is.
\end{proof}

\newcommand{\Leg}[2]{\left(\frac{#1}{#2}\right)}

In the first three lines of the following table, produced using table 15.1 of \cite{S}, we give criteria for determining the type of additive reduction $E$ has over $\pi\in A$.  In the last line of the table we define a constant $\epsilon_\pi$ which will be used in the next section.

\newcommand{\I}{\mathrm{I}}
\newcommand{\II}{\mathrm{II}}
\newcommand{\III}{\mathrm{III}}
\newcommand{\IV}{\mathrm{IV}}

\vspace{.25cm} 
\renewcommand{\arraystretch}{1.25}
\begin{equation*}
    \begin{array}{c|cc|cccc|cc} 
    \hline
    \mbox{Kodaira symbol}
    		& \I_n^* & \I_0^* & \II & \IV & \IV^* & \II^* & \III & \III^* \\
	\noalign{\hrule height 2pt}
	\ord_\pi(\Delta_\pi)
		& 6+n & 6 & 2 & 4 & 8 & 10 & 3 & 9 \\
	j\!\!\!\!\pmod{\pi} & \infty & \not\equiv\infty
		& 0 & 0 & 0 & 0 & 1728 & 1728 \\
	\epsilon_\pi & -1 & -1 & -1 & -3 & -3 & -1 & -2 & -2  \\ \hline
    \end{array}
\end{equation*}
\renewcommand{\arraystretch}{1}
\captionof{table}{Additive Reduction Information for $E/K$}
\label{table::kodaira}


\subsection{$L$-functions}\label{sec::l_fcn}

We keep the notation of the previous section and add the assumption that $j$ is non-constant, i.e.~it lies in the complement $K-\F_q$, and we remark that most of what follows extends to the case where $\Delta$ (but not necessarily $j$) is non-constant.

For each $\pi\in|U|$ and $m\geq 1$, we write $\F_{\pi^m}$ for the unique extension of $\F_\pi$ of degree $m$, $\E(\F_{\pi^m})$ for the set of $\F_{\pi^m}$-rational points of $\E/\F_\pi$, and
    $$ a_{\pi^m} = q^{m\deg(\pi)} + 1 - \#\E(\F_{\pi^m}). $$

For each $\pi\in|U|$, the zeta function of $\E/\F_\pi$ is given by the exponential generating series
\begin{equation}\label{eqn1}
	 Z(T,\E/\F_\pi) =
	 	\exp\left(\sum_{m=1}^\infty\#\E(\F_{\pi^m})\frac{T^m}{m}\right).
\end{equation}
It is a rational function in $\Q(T)$ with denominator $(1-T)(1-q^{\deg(\pi)}T)$ and numerator
    $$ L(T,\E/\F_\pi) = 1 - a_{\pi^1} T + q^{\deg(\pi)}T^2. $$

Because we assumed $j$ is non-constant, the $L$-function $L(T,E/K)$ is a polynomial in $\Z[T]$ with constant coefficient 1 (cf.~bottom of page 11 of \cite{K}) and satisfies
    $$ \deg(L(T,E/K)) = \deg(M) + 2\deg(A) - 4. $$
It has an Euler product expansion
\begin{equation}\label{eqn2}
    L(T,E/K) =
    \prod_{\pi\in|U|}L\left(T^{\deg(\pi)},\E/\F_\pi\right)^{-1}
    \cdot \!\!\prod_{\pi\in M^+}\left(1-T^{\deg(\pi)}\right)^{-1}
    \cdot \!\!\prod_{\pi\in M^-}\left(1+T^{\deg(\pi)}\right)^{-1}.
\end{equation}

Using (\ref{eqn1}) and the formal identity $1/(1-\alpha T)=\exp(\sum_{n=1}^\infty (\alpha T)^n/n)$
it is easy to show that
\begin{equation*}
	L(T^{\deg(\pi)},\E/\F_\pi) =
    		\exp\left( \sum_{n\geq 1,\deg(\pi)|n}
			\deg(\pi)\cdot a_{\pi^{n/\deg(\pi)}}\frac{T^n}{n}
		\right).
\end{equation*}
Therefore, if we define
\begin{equation*}
    b_n = \sum_{\pi\in|U|,\deg(\pi)|n} \deg(\pi)\cdot a_{\pi^{n/\deg(\pi)}}
+ \sum_{\pi\in M^+, \deg(\pi)\mid n} \deg(\pi) + 
		\sum_{\pi\in M^-,\,\deg(\pi)|n} \deg(\pi)(-1)^{n/\deg(\pi)},
\end{equation*}
then we can rewrite (\ref{eqn2}) as
\begin{equation}\label{eqn3}
    L(T,E/K) = \exp\left(
    		\sum_{n=1}^\infty b_n \frac{T^n}{n}							 
    	\right).
\end{equation}

If we truncate the formal series expansion of the right side of (\ref{eqn3}) by reducing modulo $T^{N+1}$ for $N\geq 0$, then we obtain the congruence
\begin{equation}\label{eqn4}
    L(T,E/K)\equiv  \exp\left(
    		\sum_{n=1}^N b_n \frac{T^n}{n}\right)
		\pmod{T^{N+1}}.
\end{equation}
Thus $L(T,E/K)\pmod{T^{N+1}}$ is completely determined by $\{b_n:1\leq n\leq N\}$ for $N=\deg(L(T,E/K))$, and by definition this set is determined by the Euler factors over $\pi\in|\P^1|$ satisfying $\deg(\pi)\leq N$.  In fact, by taking the functional equation into consideration, as described below, it suffices to take $N=\lfloor\deg(L(T,E/K))/2\rfloor$.

If we write $L(T,E/K)=\sum_{n=0}^N c_nT^n$ for $N=\deg(L(T,E/K))$, then to recover $c_0,\ldots,c_N$ from (\ref{eqn4}) it suffices to apply the following lemma.

\begin{lemma}\label{lemma2}
    If $\{c_0=1\}\cup\{b_n,c_n:1\leq n\leq N\}$ are numbers satisfying
    $$
    		\exp\left(\sum_{n=1}^N b_n\frac{T^n}{n}\right)
    		\equiv
    		\sum_{n=0}^N c_n T^n
		\pmod{T^{N+1}},
	$$
then they satisfy the recursive relation
    $$ c_n = \frac{1}{n}\sum_{m=1}^n b_m\cdot c_{n-m},\quad n\geq 1. $$
\end{lemma}

\begin{proof}
If we take the (formal) logarithmic derivative of both sides of the assumed relation between the $b_n$ and $c_n$ and clear denominators, then we obtain the relation
    $$ 
    	\left(\sum_{n=1}^N b_n T^n\right)\left(\sum_{n=0}^N c_n T^{n-1}\right)
    \equiv
    \sum_{n=1}^N nc_n T^{n-1}
		\pmod{T^{N+1}}. $$
The lemma follows by expanding the left side and comparing the coefficients on each side.
\end{proof}

As stated above, $L(T,E/K)$ satisfies a functional equation: there is $\varepsilon(E/K)\in\{\pm 1\}$ such that
\begin{equation}\label{eqn::fe}
	L(T,E/K)=\varepsilon(E/K)\cdot(qT)^N\cdot L(1/(q^2T),E/K),
\end{equation}
hence we have the relation
\begin{equation}\label{eqn5}
    c_n=\varepsilon(E/K)\cdot q^{2n-N}\cdot c_{N-n},\quad 0\leq n\leq N.
\end{equation}
In the following lemma we write $\Leg{\epsilon_\pi}{\pi}$ for the Legendre
symbol in $\F_\pi$ of $\epsilon_\pi$, defined in Table \ref{table::kodaira},
and give a formula for $\varepsilon(E/K)$ (cf.~corollary 5 of \cite{H}).

\begin{lemma}\label{lem::sign-fe}
	$$ \varepsilon(E/K) =
		(-1)^{\#M^+}\!\cdot\prod_{\pi\in A}\Leg{\epsilon_\pi}{\pi}.
	$$
\end{lemma}

\begin{proof}
This follows from calculations in \cite{R} where the sign is the global root number of $E/K$ and is given by a product of local root numbers.  If $\pi\in|U|$, then the local root number is trivial by proposition 8 of \loccit~with $\tau=1$.  If $\pi\in M\cup A$, then we apply parts (ii) and (iii) theorem 2 of \loccit~with $\tau=1$ for the remaining cases.  Note, if $\pi\in A$ and does not have Kodaira symbol $I_n^*$ with $n>0$, then we need the assumption that $q$ is not divisible by 2 or 3.
\end{proof}

We observe that the recursive relation given by lemma~\ref{lemma2} enables us to perform a consistency check when trying to compute $L(T,E/K)$: the $b_m$ and $c_m$ are integers, so for each $n\geq 1$, the integer $\sum_{m=1}^n b_m\cdot c_{n-m}$ must be divisible by $n$.  A second consistency check is to compute $c_n$ for one or more $n>\lfloor N/2\rfloor$ using the same method as for smaller $n$ and then to verify that (\ref{eqn5}) holds.  While one would not want to use the latter check when computing large data sets, it is very useful for making sure that the calculations are correct because one can use it to test a small subset of data.

\subsection{Quadratic Twists}\label{sec::twist}

We continue the notation of the previous sections.  Thus we fix an elliptic curve $E/K$ and a minimal Weierstrass model $y^2=x^3+ax+b$ of $E$ over $\O_K$.  For each $f\in K^\times$, we define $E_f/K$ to be the elliptic curve with affine model $y^2=x^3+f^2ax+f^3b$.


\begin{lemma}
Suppose $L/K$ is an extension.
If an elliptic curve over $K$ is $L$-isomorphic to $E$, then it is $K$-isomorphic to some $E_f$ and $\sqrt{f}\in L$.  Conversely, if $\sqrt{f}\in L$, then $E$ and $E_f$ are $L$-isomorphic.
\end{lemma}

\begin{proof}
If $y^2=x^3+a'x+b'$ is an affine model for an elliptic curve $E'/K$, then an $L$-isomorphism $E\to E'$ must take the form $(x,y)\mapsto (x/c^2,y/c^3)$ for some $c\in L^\times$ (see page 50 of \cite{S}); recall $j$ is non-constant, hence neither 0 nor 1728.  In particular, $a'=c^4a$ and $b'=c^6b$, so $f=c^2=b'/a'$ lies in $K^\times$ and $E'=E_f$.  Conversely, if $c=\pm\sqrt{f}\in L$, then $(x,y)\mapsto(x/c^2,y/c^3)$ is an $L$-isomorphism $E\to E_f$.
\end{proof}

The lemma implies $L=K(\sqrt{f})$ is the smallest extension over which $E,E_f$ are $L$-isomorphic.  Hence if $f$ lies in the complement $K^\times-(K^{\times})^2$, then $E_f/K$ is a so-called quadratic twist of $E/K$.

\begin{lemma}\label{lemma::twist}
$E_f,E_g$ are $K$-isomorphic if and only if $f=gc^2$ for some $c\in K^\times$.
\end{lemma}

\begin{proof}
Replace $E/K$ by $E_g/K$ and apply the previous lemma with $L=K$.
\end{proof}

We define the family of polynomials
$$
    \FF = \{ f\in\O_K : f\mbox{ is monic, square-free, prime to $\Delta$} \}
$$
and write $\FF_d\subset\FF$ for the subset of $f$ satisfying $\deg(f)=d$.  The previous lemma implies the $E_f$ are pairwise non-$K$-isomorphic for $f\in\FF$, while for a fixed $d$ we will see that $\deg(L(T,E_f/K))$ is independent of $f\in\FF_d$.  The latter fact would not be true if we dropped the condition that $f$ be relatively prime to $\Delta$.  As we will see, if $\alpha\in\F_q^\times$ is a non-square, then $L(T,E_{\alpha f}/K)=L(-T,E_f/K)$, hence the reason we restrict to monic $f$.

If $f\in\FF$, then one can easily verify that $y^2=x^3+f^2ax+f^3b$ is a minimal Weierstrass model for $E_f$ over $\O_K$ with discriminant $f^6\cdot\Delta$.  There is a unique integer $e$ such that the substitution $(t,x,y)\mapsto (1/u,x/u^e,y/u^e)$ yields a minimal Weierstrass model for $E_f$ over $\F_q[u]$, and we glue the models together over $\P^1-\{0,\infty\}$ to construct the minimal Weierstrass model $\E_f\to\P^1$.  We write $M_f$ and $A_f$ respectively for the divisors of multiplicative and additive reduction respectively of $\E_f\to\P^1$.

We write $\A^1$ for the complement $\P^1-\{\infty\}$.
If $\pi\in|\A^1|$, then one can easily verify that
    $$ M_f\cap\A^1=M\cap\A^1,\qquad
       A_f\cap\A^1=(A\cap\A^1)\cup\{\pi\in|\A^1|:\pi|f\}. $$
If $\pi\in M_f\cap\A^1$, then one can also easily verify that $\E_f/\F_\pi$ has the same splitting behavior as $\E/\F_\pi$ if and only if the image of $f$ is a square in $\F_\pi$, and otherwise it has the opposite splitting behavior; that is,
    $$ M_f^{\pm}\cap\A^1 =
    		\left\{\pi\in M^{\pm}\cap\A^1 : \Leg{f}{\pi}=\pm 1 \right\}
		\cup
    		\left\{\pi\in M^{\mp}\cap\A^1 : \Leg{f}{\pi}=\mp 1 \right\}.
	$$
If $f\in\FF_d$ and $d$ is even, then $\E$ and $\E_f$ are isomorphic over $\F_\infty$.  On the other hand, if $d$ is odd, then the Kodaira symbols of $\E/\F_\infty$ and $\E_f/\F_\infty$ form an unordered pair $\{S,S^*\}$, where $S\in\{\I_n,\II,\III,\IV\}$.

If we fix a non-square $\alpha\in\F_q^\times$ and $f\in\FF$, then a similar calculation for $E_{\alpha f}$ shows that the Kodaira symbols for $\E_f$ and $\E_{\alpha f}$ are the same for all $\pi\in|\P^1|$, so $M_{\alpha f}=M_f$ and $A_{\alpha f}=A_f$.  If $\pi\in A_f$, then $\E_f$ and $\E_{\alpha f}$ are isomorphic over $\F_\pi$.  On the other hand, for every $\pi\in|\P^1-A_f|$, there is a unique quadratic twist of $\E_f/\F_\pi$, which we call the scalar twist, and it is easy to show that $\E_{\alpha f}/\F_\pi$ is isomorphic to $\E_f/\F_\pi$ if $\deg(\pi)=[\F_\pi:\F_q]$ is even, otherwise it is the scalar twist.  We call $E_{\alpha f}$ the scalar twist of $E_f$.


\subsection{$L$-functions of Quadratic Twists}\label{sec::l-twist}

We continue the notation of the previous section and fix an elliptic curve
$E/K$ and a quadratic twist $E_f/K$.  If $U_f$, $M_f$, and $A_f$ are the
primes over which $E_f$ has good, multiplicative, and additive reduction
respectively, then we can use the results of section~\ref{sec::l_fcn} to
infer that the $L$-function $L(T,E_f/K)$ has Euler product expansion
\begin{equation*}
    L(T,E_f/K) =
    \prod_{\pi\in|U_f|}L\left(T^{\deg(\pi)},\E_f/\F_\pi\right)^{-1}
    \cdot \!\!\prod_{\pi\in M_f^+}\left(1-T^{\deg(\pi)}\right)^{-1}
    \cdot \!\!\prod_{\pi\in M_f^-}\left(1+T^{\deg(\pi)}\right)^{-1}.
\end{equation*}
There is an important observation which relates this to the Euler product
expansion in (\ref{eqn2}) of $L(T,E/K)$: if $\pi$ lies in $U\cap U_f$ and
if $\chi_\pi(f)\in\{\pm 1\}$ denotes the Legendre symbol of $f\pmod{\pi}$, then
\[
    L(T,\E_f/\F_\pi) = L(\chi_\pi(f)T,\E/\F_\pi).
\]
In particular, if one has precomputed sufficiently many Euler factors for
$L(T,E/K)$, then for most of the Euler factors of $L(T,E_f/K)$, the cost of
computing the factor is essentially the cost of computing $\chi_\pi(f)$.

\subsection{Pullbacks}\label{sec::pull}

We continue the notation of previous sections and fix an elliptic curve
$E/K$ and a minimal Weierstrass model $y^2=x^3+ax+b$ of $E$ over $\O_K$.
We write $M$, $A$ respectively for the subsets of primes in $K$ over which
$E$ has multiplicative and additive reduction respectively.

\newcommand{\TT}{\theta}

We fix a rational function field $L=\F_q(w)$ and a non-constant element
$\TT\in L$, and we write $\TT^*:K\to L$ for the embedding induced by sending
$t$ to $\TT$.  Let $\O_L\subset L$ be the integral closure of $\O_K$, and let
$\O_{\bar\infty}\subset L$ be the integral closure of $\O_\infty\subset K$.
We call the primes of $\O_L$ the finite primes of $L$ and the primes of
$\O_{\bar\infty}$ the infinite primes of $L$.  In general, $\O_L$ is not
$\F_q[w]$ and $\O_{\bar\infty}$ is not the local ring with uniformizer
$1/w$, but rather the infinite primes are the poles of $\TT$.  If $\pi$ is a
prime (finite or infinite) in $L$, we write $\TT(\pi)$ for the corresponding
prime in $K$, $f_\pi$ for the inertial degree of $\pi$ over $\TT(\pi)$, and
$e_\pi$ for the ramification degree.

We write $E_L$ for the elliptic curve over $L$ (eliding the dependence on the choice of $\TT$) with affine model
$y^2=x^3+\TT^*(a)x+\TT^*(b)$, and thus the coefficients of the model are
rational functions in $w$ when $a,b$ are viewed as elements of $L$ via
$\TT^*$.  A priori, this model is not a minimal Weierstrass model of $E_L$
over $\O_L$, but if $\pi$ is a finite prime that is unramified in $L$ or
if $\TT(\pi)$ does not lie in $A$, then the model is a minimal Weierstrass
model over the local ring $\O_\pi$.  Similarly, if $\pi\in\O_{\bar\infty}$
is an infinite prime, then a minimal Weierstrass model for $E$ over
$\O_\infty$ is guaranteed to be a minimal Weierstrass model over $\O_\pi$
only if $\pi$ is unramified over $\TT(\pi)$ or if $\infty$ does not lie in
$A$.

Now suppose that $\pi$ is a prime of $L$ such that $\pi$ is ramified over
$\TT(\pi)$ and $E$ has bad reduction over $\TT(\pi)$, and let $y^2=x^3+a_\pi
x+b_\pi$ be a minimal Weierstrass model of $E$ over $\O_{\TT(\pi)}$ and let
$\Delta_{\TT(\pi)}$ be the discriminant of this model.  If $E$ has Kodaira
type $\mathrm{I}_n$ over $\TT(\pi)$ and if $e=e_\pi$ ($e$ being the unique integer used to obtain a model for $E$ over $\F_q[u]$), then
$y^2=x^3+\TT^*(a_\pi)x+\TT^*(b_\pi)$ is a minimal Weierstrass model of $E_L$
over $\O_\pi$ with discriminant $\Delta_\pi=\TT^*(\Delta_{\TT(\pi)})$ and
$E_L$ has Kodaira type $\mathrm{I}_{en}$ over $\pi$.  On the other hand, if
$\TT(\pi)$ lies in $A$, then the Kodaira type of $E_L$ over $\pi$ may differ
from the Kodaira type of $E$ over $\TT(\pi)$ depending on $e=e_\pi$.  More
precisely, if $E$ has Kodaira type $\I_n^*$ over $\TT(\pi)$, then then $E_L$
has Kodaira type $\I_{en}$ or $\I_{en}^*$ over $\pi$ if $e$ is even or odd
respectively.  Otherwise, the discriminant $\Delta_\pi$ for a minimal
Weierstrass model of $E_L$ over $\O_\pi$ satisfies
$\ord_\pi(\Delta_\pi)\equiv e\cdot\ord_{\TT(\pi)}(\Delta_{\TT(\pi)})\pmod{12}$,
hence the Kodaira type of $E_L$ over $\pi$ is completely determined by
the Kodaira type of $E$ over $\TT(\pi)$ and table \ref{table::kodaira}.

Aside from the fact that one can use pullbacks to generate new elliptic
surfaces from old, the other important role they play lies in the fact that
\emph{any} elliptic curve over $L$ with non-constant $j$-invariant can
be written as a quadratic twist of the pullback of the `versal' elliptic
curve $E/K$ with affine model
\[
    y^2 = x^3 - \frac{108t}{t-1728}x + \frac{432t}{t-1728}.
\]
One can easily verify that this elliptic curve has $j$-invariant $t$, hence
for an elliptic curve over $L$ one can take $\TT$ to be the $j$-invariant and
use lemma~\ref{lemma::twist} to infer that an appropriate quadratic
twist of the pullback will be the original elliptic curve over $L$.  We remark that if $\sqrt{2}\not\in\F_q$, this model is the twist of the model in \cite[proof of prop.~III.1.4]{S} by the quadratic extension $K(\sqrt{2})/K$.  The latter model has split-multiplicative reduction at $t=\infty$ while our model forces $\varepsilon(E/K)=-1$.  In both cases the $L$-function has degree one and thus in our model we have $L(T,E/K)=1-qT$.  One can verify that the point $P=(4,8)$ lies on $E$ and has height $1/2$, thus the Mordell-Weil and analytic ranks of $E$ are both one.

\subsection{$L$-functions of Pullbacks}\label{sec::l-pull}

We continue the notation of the previous section and fix an elliptic curve
$E/K$ and a pullback $E_L/L$.  If $U_L$, $M_L$, and $A_L$ are the primes
over which $E_L$ has good, multiplicative, and additive reduction
respectively, then the $L$-function $L(T,E_L/L)$ has Euler product expansion
\begin{equation*}
    L(T,E_L/L) =
    \prod_{\pi\in|U_L|}L(T^{\deg(\pi)},\E_L/\F_\pi)^{-1}
    \cdot \!\!\prod_{\pi\in M_L^+}(1-T^{\deg(\pi)})^{-1}
    \cdot \!\!\prod_{\pi\in M_L^-}(1+T^{\deg(\pi)})^{-1}.
\end{equation*}
As in the case of quadratic twists, if one has computed enough information
for $E/K$, then it is relatively cheap to compute most of the Euler factors
of $L(T,E_L/L)$.  More precisely, if $E_L$ has good reduction over $\pi$
and if $E$ has good reduction over $\mathfrak{p}=\TT(\pi)$, then
\[
    L(T,\E_L/\F_\pi) = 1 - a_{\mathfrak{p}^f}T + q^{\deg(\pi)}T^2,
\]
where $f_\pi$ is the inertia degree of $\pi$ over $\mathfrak{p}$ and
$a_{\mathfrak{p}^{f_\pi}}$ can be determined by the expansion (\ref{eqn1}).  In
practice, it is easier to keep track of $a_{\mathfrak{p}^n}$ for several $n$
rather than use (\ref{eqn1}) directly because, among other reasons, it is
difficult to compare elements of $\F_{\mathfrak{p}}$ with the corresponding
subfield of $\F_\pi=\F_{\mathfrak{p}^{f_\pi}}$.  Nonetheless, the additional work
one must do is small since `most' elements of $\F_{q^n}$ have degree $n$
over $\F_q$, and the upshot is that most of the work of computing $b_n$ in
the corresponding expansion (\ref{eqn3}) for $L(T,E_L/L)$ is the cost of
explicitly evaluating the map $\theta:\P^1(\F_{q^n})\to\P^1(\F_{q^n})$ for all
elements in the domain.


\section{\bigellff}\label{sec::ellff}

The discussion in section \ref{sec::l_fcn} above naturally leads to a na\"ive algorithm to compute the $L$-function of a non-isotrivial elliptic curve defined over a function field via counting points on a finite number of its fibers. Moreover, if sufficiently many Euler factors have been computed for the versal elliptic curve, one can realize a given elliptic curve as a pullback and quadratic twist and use the results of section \ref{sec::twist} -- \ref{sec::l-pull} to significantly reduce the number of fibers on which one needs to count points. The authors have written a standalone \cpp~ library called \ellff~ built upon Shoup's Number Theory Library \cite{NTL}, which can be added as a module to the free open-source mathematics software system Sage \cite{Sage}. The package allows anyone to efficiently compute these $L$-functions on their own.\footnote{The library currently only allows for characteristic not 2 or 3, though handling these cases is straightforward and will be addressed in a future release.}

Internally, the library uses tables, computed on demand, to represent Euler factors.  If one asks for the $L$-function of a curve, then the library demands the minimal number of tables necessary.  One may also demand and manipulate the tables directly, e.g.~in order to study how the sizes of special fibers vary.  The library uses a database in order to reduce the complexity of computing a table, e.g.~by returning a previously calculated copy of the table or twisting a table for another curve with the same $j$-invariant.  A user whose database has the appropriate tables for the versal curve will benefit from such reductions, and thus we have made available a modest collection for download.  The user may also save their own tables in the database in order to facilitate calculating tables for families of curves.  For more information and setup instructions, see \texttt{http://ellff.sagemath.org}.



\section{Computations}\label{sec::compute}

Using the discussion from section \ref{sec::ec}, a database of $L$-functions was amassed for the family of quadratic twists of the following four elliptic curves (with notation consistent with that found in \cite{MP}):
\begin{equation*}\label{x222}
X_{222}: y^2  = x^3-27(t^4-t^3+t^2)x + 27(2t^6-3t^5-3t^4+2t^3),
\end{equation*}
\begin{equation*}\label{x211}
X_{211}: y^2 = x^3 - 27t^4x+54t^5(t-2),
\end{equation*}
\begin{equation*}\label{x321}
X_{321}: y^2 = x^3 - 108t^3(4t-3)x + 432t^5(8t-9),
\end{equation*}
\begin{equation*}\label{x431}
X_{431}: y^2 = x^3 - t^3(27t-24)x + t^4(54t^2-72t+16).
\end{equation*}
These are the only elliptic curves, up to isogeny, over $\F_q(t)$ with $(q,6)=1$ such that $\#M=2$ and $\#A=1$. They are normalized so that $\infty\in M^+$, $t\in A$, and $t-1\in M$, forcing the $L$-function to be trivial for each of the curves. Note that the first curve is the Legendre curve,\footnote{Strictly speaking, $X_{222}$ is a twist by $-t$ of the usual Legendre curve model $y^2=x(x-1)(x-t)$.} given by the alternative model
\begin{equation*}\label{e-leg}
X_{222}: y^2 = x(x+t)(x+t^2).
\end{equation*}

For each of these curves, we considered all prime $q\in\QQ=\{5,7,\dots,29\}$ and computed the $L$-functions of all the twists with bounded degree. The bound on the degree was determined by considerations on computational feasibility and depended on the size of the field of constants. The following table lists the number of twists over $\F_q$ of degree $d$ in each family $\FF_d$. This number does not depend on which of the four curves above one considers. A blank entry in the table denotes that the $L$-functions for all twists for the given $d$ and $q$ were not determined due to the computation requiring an excessive amount of time.

\begin{table}[hbt]
\renewcommand{\arraystretch}{1.25}
\begin{center}
\begin{tabular}{cr@{\hspace{.4cm}}r@{\hspace{.4cm}}r@{\hspace{.4cm}}r@{\hspace{.4cm}}r@{\hspace{.4cm}}r@{\hspace{.4cm}}r@{\hspace{.4cm}}r}\hline
 & 5 & 7 & 11 & 13 & 17 & 19 & 23 & 29 \\ \noalign{\hrule height 2pt}
$\#\FF_1$ & 3 & 5 & 9 & 11 & 15 & 17 & 21 & 27 \\
$\#\FF_2$ & 13 & 31 & 91 & 133 & 241 & 307 & 463 & 757 \\
$\#\FF_3$ & 71 & 227 & 1019 & 1751 & 4127 & 5867 & 10691 & 22007 \\
$\#\FF_4$ & 345 & 1573 & 11181 & 22729 & 70113 & 111421 & 213762 & 638121 \\
$\#\FF_5$ & 1739 & 11033 & 123029 & 295523 \\
$\#\FF_6$ & 8677 & 77203 \\
$\#\FF_7$ & 43407 \\
$\#\FF_8$ & 217009 \\
$\#\FF_9$ & 1085075 \\ \noalign{\hrule height 2pt}
All & 1356339 & 90072 & 135329 & 320147 & 74496 & 117612 & 224937 & 660912 \\ \hline
\end{tabular}
\caption{Number of Twists in $\FF_d$ for $q\in\QQ$}
\label{table::fd-size}
\end{center}
\end{table}

\subsection{Goldfeld's Conjecture}

In 1979 Goldfeld \cite{G} conjectured an average value for the analytic rank of a family of quadratic twists of a fixed elliptic curve $E/\Q$:
\begin{conjecture}[Goldfeld]\label{conj::goldfeld}
For $D$ a discriminant,
\begin{equation}\label{g-eqn}
\lim_{D\rightarrow\infty} \frac{\sum_{|d|<D} r(E_d)}{\#\{d:|d|<D\}} = \frac{1}{2}
\end{equation}
where $r(E_d)$ is the order of vanishing at $s=1$ of the $L$-function of the quadratic twist $E_d/\Q$.
\end{conjecture}

Goldfeld's conjecture concerns the analytic rank of an elliptic curve, though it is important to note that many authors replace the analytic rank with the algebraic rank (i.e. the rank, as a free $\Z$-module, of the group $E(K)$ of $K$-rational points on $E$ modulo torsion), invoking the Birch and Swinnerton-Dyer conjecture if needed. For a survey of results on the average value and variation of the (algebraic) ranks of elliptic curves in a family of quadratic twists in the number field setting, see \cite{RS}. A more recent paper \cite{BMSW} provides data for the average value and distribution of the analytic ranks of elliptic curves over $\Q$ \emph{ordered by conductor}. Thus the reader should be wary of concluding that the data presented therein either supports or undermines Goldfeld's conjecture, which considers the family of quadratic twists of a fixed elliptic curve and not all elliptic curves with bounded conductor.

Goldfeld's conjecture has a direct analog in the function field setting: for an elliptic curve $E$ over $K$, we set its analytic rank $r$ to be the order of vanishing of $L(T,E/K)$ at $T=1/q$. Instead of considering all twists by $d$ with $|d|<D$, we consider those twists in $\FF^*_D = \bigcup_{d\leq D} \FF_d$ and let $D$ grow to infinity as before:
\begin{conjecture}\label{conj::goldfeld-ff}
For $D$ a positive number,
\begin{equation}
\lim_{D\rightarrow\infty} \frac{\sum_{f\in \FF_D^*} r(E_f)}{\#\FF_D^*} = \frac{1}{2}
\end{equation}
where $r(E_f)$ is the order of vanishing at $s=1$ of the $L$-function of the quadratic twist $E_f/\F_q(t)$.
\end{conjecture}

One would like a lower bound on the average analytic rank over the family of interest $\FF_d$ analogous to that found in Proposition 1 on p. 114 of \cite{G}. Unlike this proposition where the average is taken over all discriminants, determining the average over $\FF_d$ is non-trivial. But using the functional equation, it is clear that if in the limit the average of the sign of the functional equation over $\FF_d$ is 0, the average analytic rank over $\FF_d$ is at least $1/2$. We next prove such a lower bound using this line of argument.

We begin by letting $M\cap\A^1=\{\pi_1,\dots\pi_r\}$ be the finite primes where $E/K$ has multiplicative reduction and setting $N=\pi_1\cdots\pi_r$. 

\begin{prop}
There exists $\vareps_d\in\{\pm 1\}$ such that for all $f\in\FF_d$,
\begin{equation}\label{eqn::sign-twist}
\vareps(E_f/K) = \vareps_d \cdot \vareps(E/K) \cdot \Leg{f}{N},
\end{equation}
where $\Leg{\cdot}{N}$ is the Jacobi symbol of $N$.
\end{prop}

\begin{proof}
We proceed by examining the contribution to the sign from the places of bad reduction.

\emph{Case 1:} If $\pi\in A_1\cap\A^1$, $E/K$ and $E_f/K$ have the same Kodaira type at $\pi$. Thus there is no change to the local contribution from $\vareps(E/K)$ to $\vareps(E_f/K)$ for such $\pi$.

\emph{Case 2:} If $\pi$ is a finite prime that divides $f$, then $E_f$ has type $I_0^*$ reduction over $\pi$.
Thus the contribution to the sign in this case is given by
$$\Leg{\epsilon_{\pi,f}}{\pi}=\Leg{-1}{\pi}\equiv q^{\deg\pi}\pmod{4},$$
implying the total contribution $\epsilon_f$ to the sign coming from those $\pi\in A_f-A$ satisfies
$$\epsilon_f\equiv q^d\pmod{4}.$$
Thus the change in the local contribution from $\vareps(E/K)$ to $\vareps(E_f/K)$ from these primes depends only on $d$.

\emph{Case 3:} For $\pi\in M_f\cap\A^1=M\cap\A^1$, the splitness at $\pi$ changes if and only if $\Leg{f}{\pi}=-1$. Thus the total change in the local contribution from $\vareps(E/K)$ to $\vareps(E_f/K)$ is $\Leg{f}{N}$.

\emph{Case 4:} For $\pi=\infty$, the reduction of $E/K$ and $E_f/K$ are the same if $d$ is even by the discussion in section \ref{sec::twist}. If $d$ is odd, then the reduction depends only the leading coefficient of $f$. Thus the local contribution to the sign is independent of $f\in\FF_d$, so the change in the local contribution from $\vareps(E/K)$ to $\vareps(E_f/K)$ for $\pi=\infty$ depends only on $d$.

These cases exhaust all possible changes to the sign of the functional equation introduced by twisting, yielding equation \ref{eqn::sign-twist}. \end{proof}

\begin{corollary}\label{cor::sign-jac}
$$\frac{1}{\#\FF_d}\sum_{f\in\FF_d} \vareps(E_f/K) = \frac{\vareps_d\cdot\vareps(E/K)}{\#\FF_d} \sum_{f\in\FF_d}\Leg{f}{N}.$$
\end{corollary}

Corollary \ref{cor::sign-jac} reduces the average of the sign of the functional equation to the average of the Jacobi symbol over $\FF_d$. The following proposition is due to private correspondence with E. Kowalski:
\begin{prop}[Kowalski]
With notation as above, we have 
$$\lim_{d\rightarrow\infty}\frac{\sum_{f\in\FF_d} \Leg{f}{N}}{\#\FF_d} = 0.$$
\end{prop}

\begin{proof}
Unless stated otherwise, we write $f,g,h\in\F_q[t]$ for arbitrary monic polynomials.  Write $\Delta=NN'$ and let $\chi_\Delta(f)$ be the characteristic function of those $f$ that are square-free and coprime to $\Delta$. Setting
$$ A_d = \sum_{f\in\FF_d} \Leg{f}{N},$$
we then have
$$
	A_d = \!\!\!\!\!
		\sum_{\scriptsize \deg(f)=d} \chi_\Delta(f)\Leg{f}{N}
		\quad \text{and} \quad
		\#\FF_d =\!\!\!\!\! \sum_{\scriptsize \deg(f)=d} \chi_\Delta(f).
$$
Let $\mu(\cdot)$ be the M\"obius function for polynomials.  If $\deg(g)>0$, then $\sum_{h\mid g}\mu(h)=0$, and otherwise $\sum_{h\mid g}\mu(h)=1$.  Thus $f\mapsto \sum_{g^2\mid f}\mu(g)$ and $f\mapsto \sum_{h\mid(\Delta,f)} \mu(h)$ are the characteristic functions for square-free polynomials and polynomials coprime with $\Delta$ respectively, and hence
$$
	\chi_\Delta(f)
		= \sum_{\scriptsize g^2\mid f} \mu(g)
		  \sum_{h\mid(\Delta,f)} \mu(h).
$$
Note, if $(g,\Delta)\neq 1$, then the right sum over $h$ vanishes, hence we can restrict to $g$ such that $(g,\Delta)=1$.  In particular, if we substitute into the above expression for $A_d$ and rearrange terms we have
\begin{eqnarray*}
	A_d & = &
		\sum_{h\mid\Delta}\,\, \mu(h)\!\!
		\sum_{\deg(g)\leq \frac{d}{2},\ (g,\Delta)= 1} \mu(g)
		\sum_{g^2h\mid f,\ \deg(f)=d} \Leg{f}{N}.
\end{eqnarray*}
If we write $f=f_1g^2h$ in the innermost sum, then we have
$$
	A_d =
		\sum_{h\mid\Delta}\,\, \mu(h)\Leg{h}{N}
		\sum_{\deg(g)\leq\frac{d}{2},\ (g,\Delta)= 1} \mu(g)
		\sum_{\deg(f_1)=d-2\deg(g)-\deg(h)} \Leg{f_1}{N}.
$$
Moreover, if we write $B_e$ for the sum $B_e=\sum_{\deg(f)=e}\Leg{f}{N}$ and if we suppose $e\geq\deg(N)$, then 
$$
	B_e
		\ \ =\ \sum_{\alpha\in\F_q[t]/(N)} \Leg{\alpha}{N}
		\sum_{\scriptsize\begin{array}{c}\deg(f)=e\\f\equiv\alpha(\text{mod }N)\end{array}}\!\!\! \!\! 1\ 
		= \sum_{\alpha\in\F_q[t]/(N)} \Leg{\alpha}{N}
		q^{e-\deg(N)}
		\ \ =\  0
	$$
(because the last sum is a complete character sum).  Therefore if we write $e=d-2\delta-\deg(h)$ and suppose $e<\deg(N)$ we have
\begin{equation*}
	A_d =
		\sum_{h\mid\Delta}\,\, \mu(h)\Leg{h}{N}
		\sum_{\frac{1}{2}(d-\deg(N)-\deg(h))\leq\delta\leq\frac{d}{2}}\ \ 
		\sum_{\deg(g)=\delta,\ (g,\Delta)= 1} \mu(g)\  B_{d-2\delta-\deg(h)}.
\end{equation*}
Observe that for any $e,\delta\geq 0$, we have $|B_e|\leq q^e$ and $(\sum_{\deg(g)=\delta} 1)\leq q^\delta$, thus for $d\geq 1$
\begin{eqnarray*}
	|A_d| & \leq &
		\sum_{h\mid\Delta}\,\, \mu(h)
		\sum_{\frac{1}{2}(d-\deg(N)-\deg(h))\leq\delta\leq\frac{d}{2}} q^{d-2\delta-\deg(h)}
		\sum_{\deg(g)=\delta,\ (g,\Delta)= 1} 1 \\
	& \leq  &
		\sum_{h\mid\Delta}\,\, \mu(h)
		\sum_{\frac{1}{2}(d-\deg(N)-\deg(h))\leq\delta\leq\frac{d}{2}}\ \ q^{\deg(N)}
		\, q^{d/2} \\
	& \ll & q^{\deg(N)+d/2}
\end{eqnarray*}
where the implied constant depends on $\Delta$ and $N$. On the other hand, $\#\FF_d \gg q^d$, so 
$$\lim_{d\rightarrow\infty}\frac{|A_d|}{\#\FF_d} \ll \lim_{d\rightarrow\infty}q^{\deg(N)-d/2} = 0,$$
proving the propsition.
\end{proof}

This proposition then leads to the desired corollary:

\begin{corollary}
With notation as above, we have
$$\lim_{D\rightarrow\infty} \frac{\sum_{f\in \FF_D^*} r(E_f)}{\#\FF_D^*} \geq \frac{1}{2}.$$
\end{corollary}

\subsection{Average Analytic Rank Data}

We define
$$ \mu(E,D) = \frac{\sum_{f\in \FF_D^*} r(E_f)}{\#\FF_D^*} $$
to be the average rank of the family of quadratic twists of $E$ up to degree $D$. This value was calculated for the four elliptic curves discussed above with increasing $D$, and the data is presented in tables \ref{table::rkavg-0}--\ref{table::rkavg-3} below, where the dependence of the average rank on $D$ is made explicit. As in the case of Table \ref{table::fd-size}, an empty entry denotes that those computations were not done. 

\begin{table}[h]
 \renewcommand{\arraystretch}{1.25}
 \begin{minipage}{0.5\textwidth}
  \begin{center}\footnotesize{
   \begin{tabular}{c@{\hspace{4pt}}c@{\hspace{4pt}}c@{\hspace{4pt}}c@{\hspace{4pt}}c@{\hspace{4pt}}c@{\hspace{4pt}}c@{\hspace{4pt}}c@{\hspace{4pt}}c@{\hspace{4pt}}c}\hline
$D$ & 5 & 7 & 11 & 13 & 17 & 19 & 23 & 29 \\ \noalign{\hrule height 2pt}
$1$ & 1.000 & 0.400 & 0.667 & 0.636 & 0.733 & 0.588 & 0.571 & 0.704 \\
$2$ & 0.688 & 0.667 & 0.680 & 0.674 & 0.668 & 0.679 & 0.661 & 0.652 \\
$3$ & 0.644 & 0.669 & 0.622 & 0.629 & 0.610 & 0.607 & 0.588 & 0.576 \\
$4$ & 0.653 & 0.659 & 0.638 & 0.620 & 0.605 & 0.599 & 0.588 & 0.575 \\
$5$ & 0.666 & 0.633 & 0.590 & 0.581 & & & & \\
$6$ & 0.628 & 0.609 & & & & & & \\
$7$ & 0.623 & & & & & & & \\
$8$ & 0.592 & & & & & & & \\
$9$ & 0.582 & & & & & & & \\ \hline
   \end{tabular}}
   \caption{$\mu(X_{222},D)$ for $q\in\QQ$}
   \label{table::rkavg-0}
  \end{center}
 \end{minipage}
 \begin{minipage}{0.5\textwidth}
  \begin{center}\footnotesize{
  \begin{tabular}{c@{\hspace{4pt}}c@{\hspace{4pt}}c@{\hspace{4pt}}c@{\hspace{4pt}}c@{\hspace{4pt}}c@{\hspace{4pt}}c@{\hspace{4pt}}c@{\hspace{4pt}}c@{\hspace{4pt}}c}\hline
$D$ & 5 & 7 & 11 & 13 & 17 & 19 & 23 & 29 \\ \noalign{\hrule height 2pt}
1 & 0.333 & 0.400 & 0.444 & 0.636 & 0.467 & 0.588 & 0.571 & 0.481 \\
2 & 0.688 & 0.556 & 0.540 & 0.549 & 0.527 & 0.568 & 0.562 & 0.545 \\
3 & 0.598 & 0.601 & 0.533 & 0.579 & 0.536 & 0.562 & 0.526 & 0.524 \\
4 & 0.662 & 0.562 & 0.543 & 0.534 & 0.529 & 0.532 & 0.525 & 0.521 \\
5 & 0.586 & 0.565 & 0.525 & 0.538 & & & & \\
6 & 0.634 & 0.539 & & & & & & \\ 
7 & 0.554 & & & & & & & \\ 
8 & 0.581 & & & & & & & \\
9 & 0.535 & & & & & & & \\ \hline
  \end{tabular}}
  \caption{$\mu(X_{211},D)$ for $q\in\QQ$}
  \label{table::rkavg-1}
  \end{center}
 \end{minipage}
\end{table}

\begin{table}[h]
 \renewcommand{\arraystretch}{1.25}
 \begin{minipage}{0.5\textwidth}
  \begin{center}\footnotesize{
   \begin{tabular}{c@{\hspace{4pt}}c@{\hspace{4pt}}c@{\hspace{4pt}}c@{\hspace{4pt}}c@{\hspace{4pt}}c@{\hspace{4pt}}c@{\hspace{4pt}}c@{\hspace{4pt}}c@{\hspace{4pt}}c}\hline
$D$ & 5 & 7 & 11 & 13 & 17 & 19 & 23 & 29 \\ \noalign{\hrule height 2pt}
1 & 0.333 & 0.400 & 0.444 & 0.636 & 0.600 & 0.471 & 0.571 & 0.481 \\ 
2 & 0.562 & 0.556 & 0.540 & 0.618 & 0.590 & 0.580 & 0.587 & 0.585 \\
3 & 0.690 & 0.570 & 0.577 & 0.605 & 0.583 & 0.565 & 0.552 & 0.558 \\
4 & 0.625 & 0.609 & 0.569 & 0.574 & 0.559 & 0.555 & 0.551 & 0.543 \\
5 & 0.618 & 0.571 & 0.553 & 0.554 & & & & \\
6 & 0.602 & 0.569 & & & & & & \\
7 & 0.587 & & & & & & & \\
8 & 0.568 & & & & & & & \\ 
9 & 0.556 & & & & & & & \\ \hline
   \end{tabular}}
   \caption{$\mu(X_{321},D)$ for $q\in\QQ$}
   \label{tabel::rkavg-2}
  \end{center}
 \end{minipage}
 \begin{minipage}{0.5\textwidth}
  \begin{center}\footnotesize{
  \begin{tabular}{c@{\hspace{4pt}}c@{\hspace{4pt}}c@{\hspace{4pt}}c@{\hspace{4pt}}c@{\hspace{4pt}}c@{\hspace{4pt}}c@{\hspace{4pt}}c@{\hspace{4pt}}c@{\hspace{4pt}}c}\hline
$D$ & 5 & 7 & 11 & 13 & 17 & 19 & 23 & 29 \\ \noalign{\hrule height 2pt}
1 & 0.333 & 0.800 & 0.444 & 0.636 & 0.467 & 0.588 & 0.571 & 0.556 \\ 
2 & 0.562 & 0.611 & 0.640 & 0.618 & 0.613 & 0.611 & 0.616 & 0.614 \\
3 & 0.621 & 0.646 & 0.602 & 0.632 & 0.580 & 0.594 & 0.567 & 0.557 \\
4 & 0.616 & 0.617 & 0.593 & 0.590 & 0.576 & 0.575 & 0.562 & 0.553 \\
5 & 0.592 & 0.600 & 0.558 & 0.555 & & & & \\
6 & 0.601 & 0.574 & & & & & & \\
7 & 0.575 & & & & & & & \\
8 & 0.568 & & & & & & & \\ 
9 & 0.548 & & & & & & & \\ \hline
   \end{tabular}}
   \caption{$\mu(X_{431},D)$ for $q\in\QQ$}
   \label{table::rkavg-3}
  \end{center}
 \end{minipage}
\end{table}

Considering each table separately, the individual columns present the data pertaining to Goldfeld's conjecture. In particular, for the largest data sets with $q=5$, there is a slow convergence to the conjecture value of $1/2$. On the other hand, each row of a table presents data relevant to Katz-Sarnak \cite{KS}, where one lets $q$ grow to infinity to determine that the conjugacy classes of the Frobenius automorphism are equidistributed in the special orthogonal group of $N\times N$ matrices with respect to Haar measure, where $N$ is the degree of the $L$-function.

We can also consider how the average ranks varies between each of the four curves for a fixed $q$, as presented in figures \ref{fig::avgrk-5} through \ref{fig::avgrk-29} below. Recall that the four cures are not isogenous but have nearly the same reduction types. Even for the smallest data sets ($q\in\{17,19,23,29\}$), there is good numerical evidence that the average ranks for each of the four curves are converging to the same value for any given $q$. Again, $q=5$ provides the strongest evidence that this value is $1/2$. Note also that in general the average rank of the Legendre curve $X_{222}$ dominates the rank of the other three curves. 

\begin{figure}[h]
\begin{center}
  \psfrag{X222}{$X_{222}$}\psfrag{X211}{$X_{211}$}\psfrag{X321}{$X_{321}$}\psfrag{X431}{$X_{431}$}\psfrag{mu}{$\mu$}\psfrag{D}{$D$}
  \includegraphics[trim = 0pt 58pt 0pt 10pt, clip, angle=-90, scale=.6]{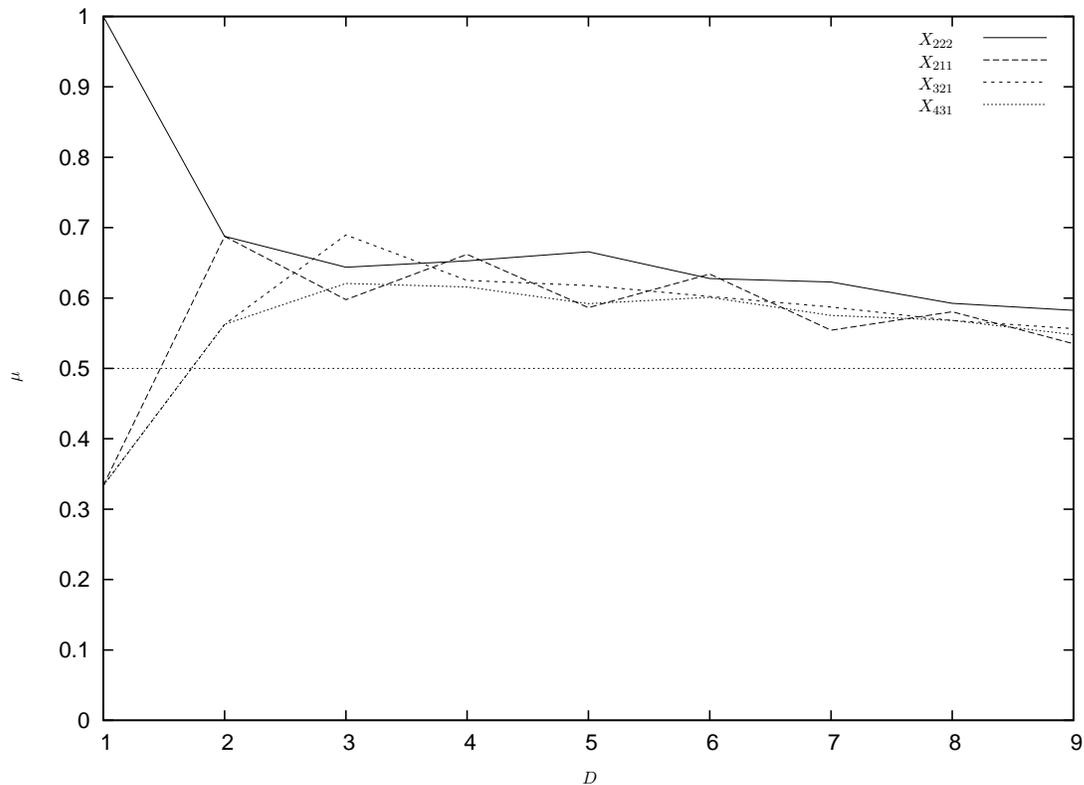}
\vspace{-.3cm} 
\caption{Variation of $\mu(X_i,D)$ as $i$ varies for $q=5$}
\label{fig::avgrk-5}
\vspace{-.4cm} 
\end{center}
\end{figure}

\clearpage

\begin{figure}[t]
\begin{center}
  \psfrag{X222}{$X_{222}$}\psfrag{X211}{$X_{211}$}\psfrag{X321}{$X_{321}$}\psfrag{X431}{$X_{431}$}\psfrag{mu}{$\mu$}\psfrag{D}{$D$}
  \includegraphics[trim = 0pt 58pt 0pt 10pt, clip, angle=-90, scale=.6]{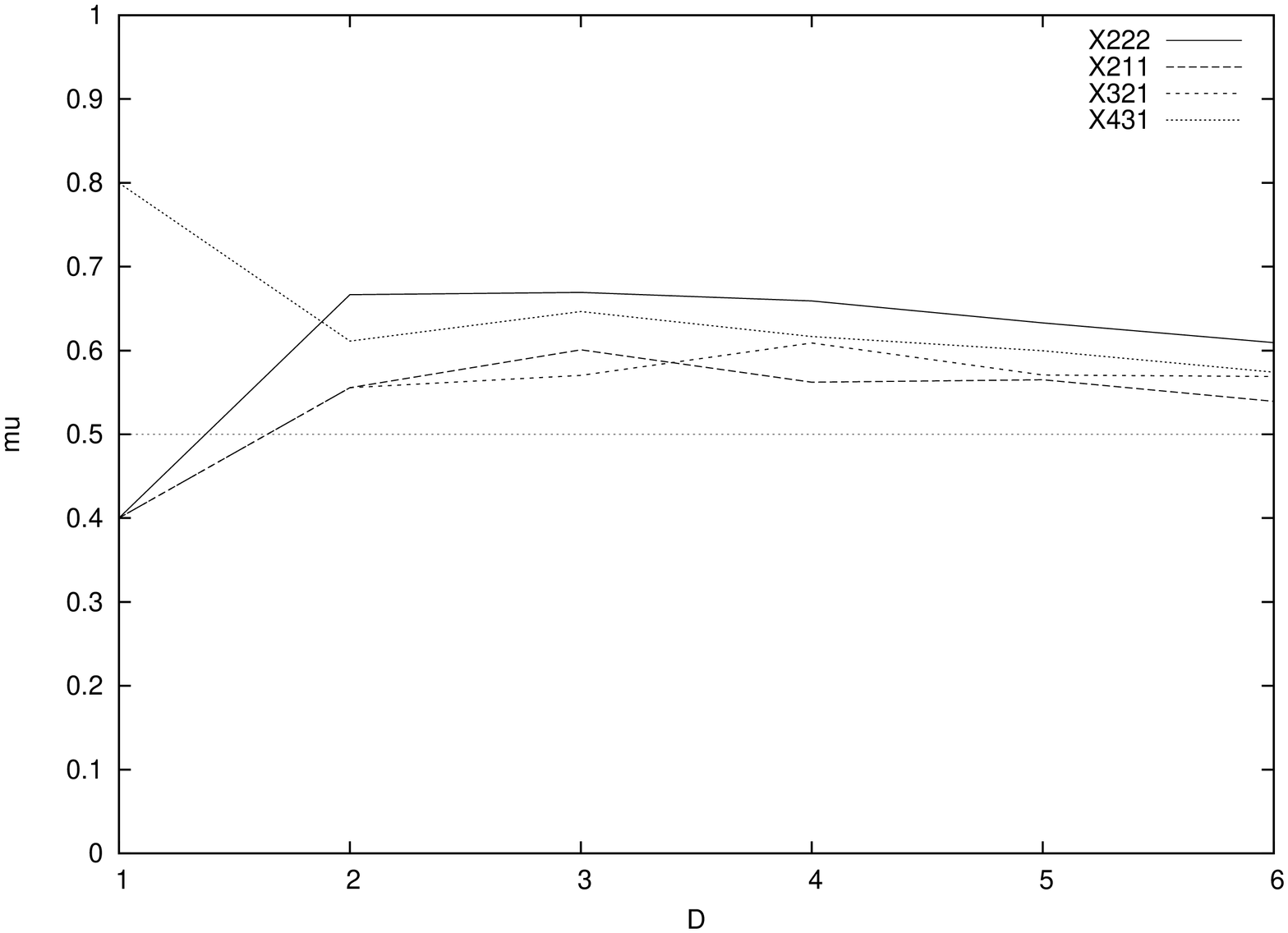}
\vspace{-.3cm} 
\caption{Variation of $\mu(X_i,D)$ as $i$ varies for $q=7$}
\label{fig::avgrk-7}
\vspace{-.7cm} 
\end{center}
\end{figure}

\begin{figure}[t]
\begin{center}
  \psfrag{X222}{$X_{222}$}\psfrag{X211}{$X_{211}$}\psfrag{X321}{$X_{321}$}\psfrag{X431}{$X_{431}$}\psfrag{mu}{$\mu$}\psfrag{D}{$D$}
  \includegraphics[trim = 0pt 58pt 0pt 10pt, clip, angle=-90, scale=.6]{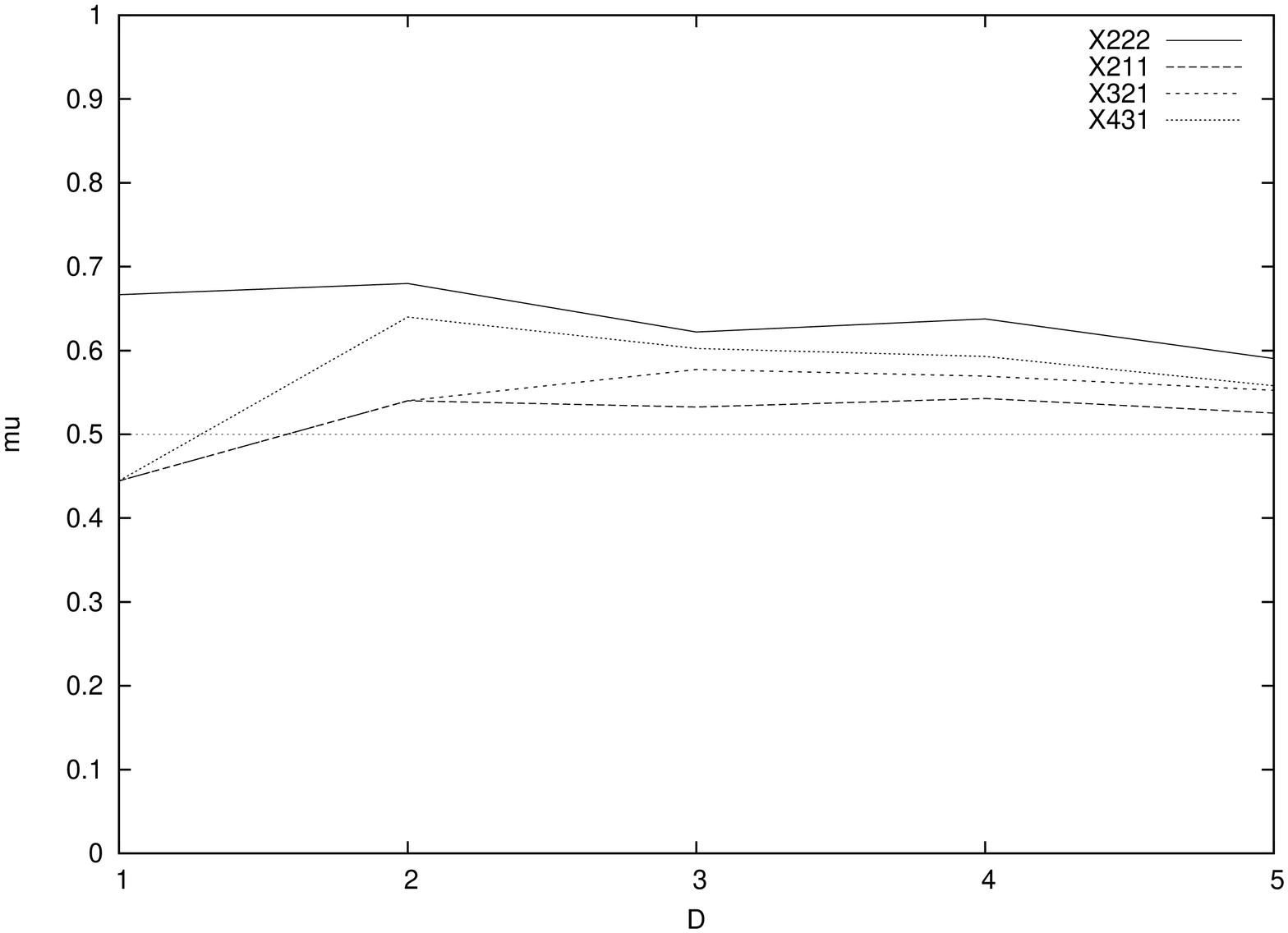}
\vspace{-.3cm} 
\caption{Variation of $\mu(X_i,D)$ as $i$ varies for $q=11$}
\label{fig::avgrk-11}
\vspace{-.4cm} 
\end{center}
\end{figure}

\clearpage

\begin{figure}[b]
\begin{center}
  \psfrag{X222}{$X_{222}$}\psfrag{X211}{$X_{211}$}\psfrag{X321}{$X_{321}$}\psfrag{X431}{$X_{431}$}\psfrag{mu}{$\mu$}\psfrag{D}{$D$}
  \includegraphics[trim = 0pt 58pt 0pt 10pt, clip, angle=-90, scale=.6]{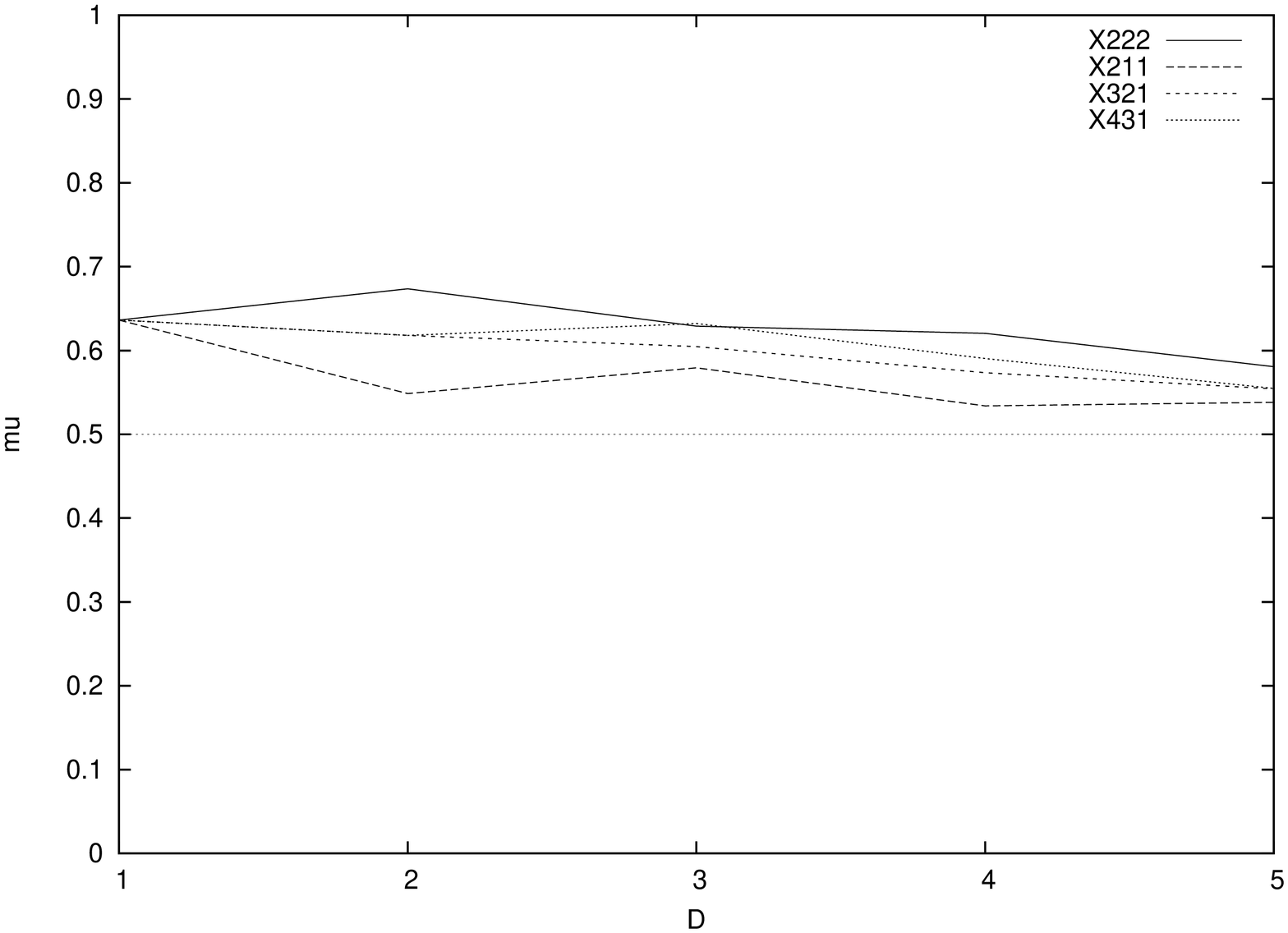}\vspace{-.4cm} 
\caption{Variation of $\mu(X_i,D)$ as $i$ varies for $q=13$}
\vspace{-.7cm} 
\label{fig::avgrk-13}
\end{center}
\end{figure}

\begin{figure}[t]
\begin{center}
  \psfrag{X222}{$X_{222}$}\psfrag{X211}{$X_{211}$}\psfrag{X321}{$X_{321}$}\psfrag{X431}{$X_{431}$}\psfrag{mu}{$\mu$}\psfrag{D}{$D$}
  \includegraphics[trim = 0pt 58pt 0pt 10pt, clip, angle=-90, scale=.6]{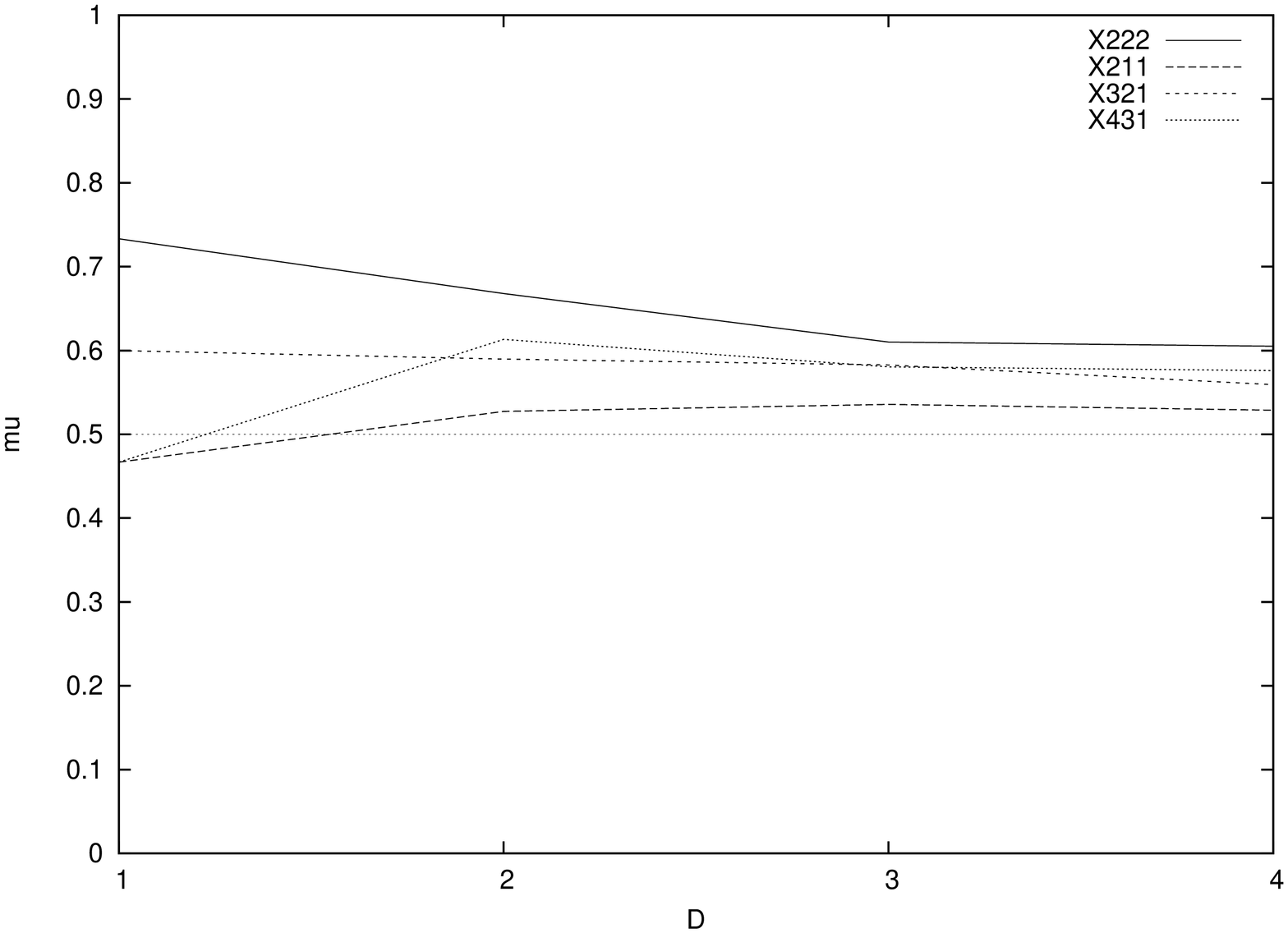}
\vspace{-.3cm} 
\caption{Variation of $\mu(X_i,D)$ as $i$ varies for $q=17$}
\label{fig::avgrk-17}
\vspace{-.4cm} 
\end{center}
\end{figure}

\clearpage

\begin{figure}[b]
\begin{center}
  \psfrag{X222}{$X_{222}$}\psfrag{X211}{$X_{211}$}\psfrag{X321}{$X_{321}$}\psfrag{X431}{$X_{431}$}\psfrag{mu}{$\mu$}\psfrag{D}{$D$}
  \includegraphics[trim = 0pt 58pt 0pt 10pt, clip, angle=-90, scale=.6]{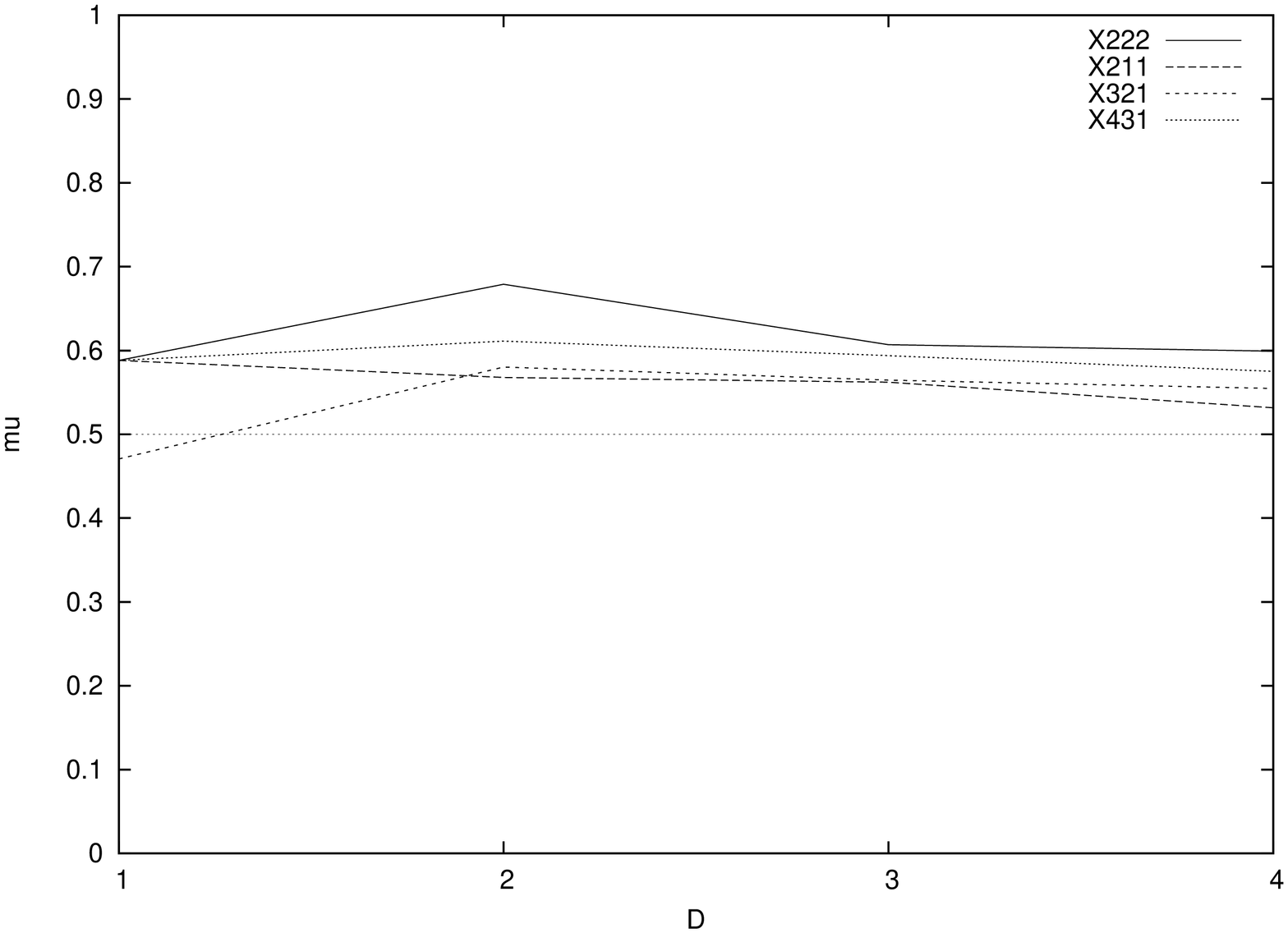}
\vspace{-.3cm} 
\caption{Variation of $\mu(X_i,D)$ as $i$ varies for $q=19$}
\label{fig::avgrk-19}
\vspace{-.7cm} 
\end{center}
\end{figure}

\begin{figure}[t]
\begin{center}
  \psfrag{X222}{$X_{222}$}\psfrag{X211}{$X_{211}$}\psfrag{X321}{$X_{321}$}\psfrag{X431}{$X_{431}$}\psfrag{mu}{$\mu$}\psfrag{D}{$D$}
  \includegraphics[trim = 0pt 58pt 0pt 10pt, clip, angle=-90, scale=.6]{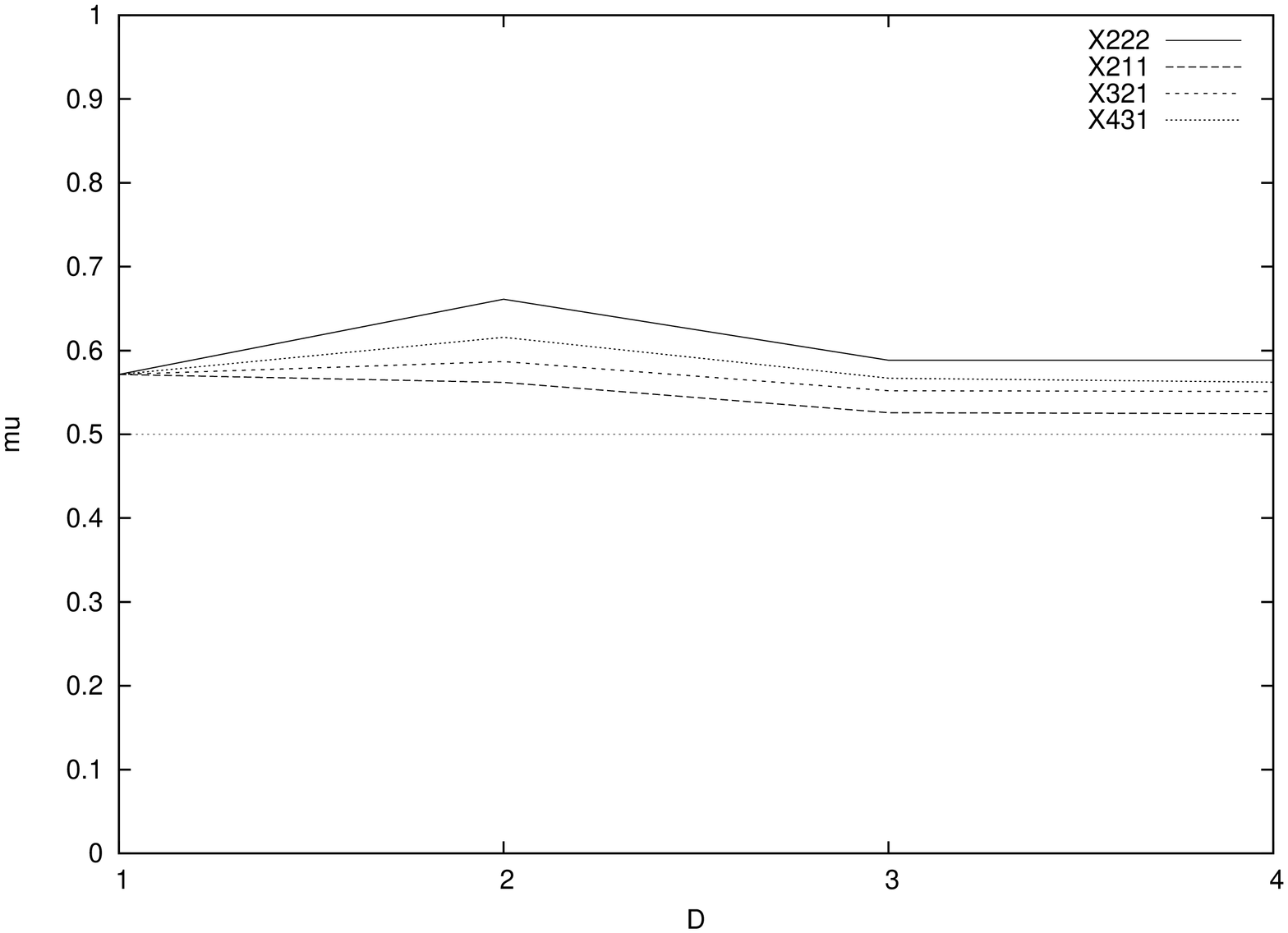}
\vspace{-.3cm} 
\caption{Variation of $\mu(X_i,D)$ as $i$ varies for $q=23$}
\label{fig::avgrk-23}
\vspace{-.4cm} 
\end{center}
\end{figure}

\clearpage

\begin{figure}[t]
\begin{center}
  \psfrag{X222}{$X_{222}$}\psfrag{X211}{$X_{211}$}\psfrag{X321}{$X_{321}$}\psfrag{X431}{$X_{431}$}\psfrag{mu}{$\mu$}\psfrag{D}{$D$}
  \includegraphics[trim = 0pt 58pt 0pt 10pt, clip, angle=-90, scale=.6]{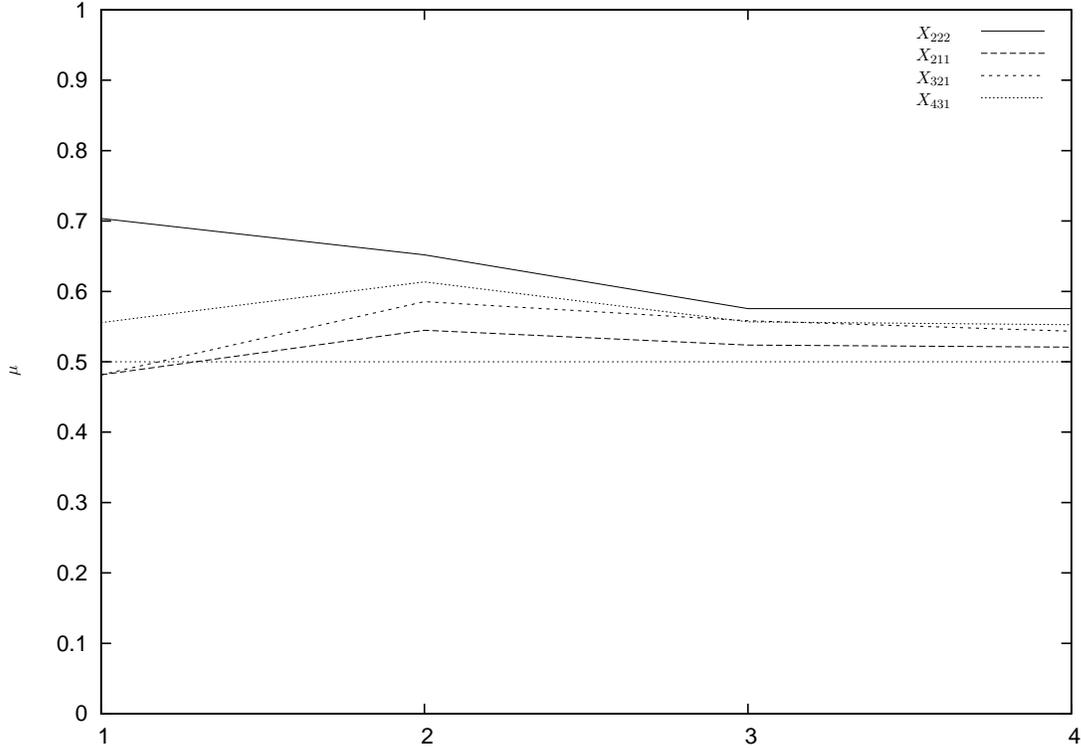}\vspace{-.4cm} 
\caption{Variation of $\mu(X_i,D)$ as $i$ varies for $q=29$}
\vspace{-.7cm} 
\label{fig::avgrk-29}
\end{center}
\end{figure}

\subsection{Rank Distributions}

Combining Goldfeld's conjecture with the parity conjecture leads to a conjecture on the density of ranks in a family of quadratic twists (for details of the formulation, see section 7.6 in \cite{RS}):
\begin{conjecture}
With notation as in Conjecture \ref{conj::goldfeld-ff},
\begin{equation*}
\lim_{D\rightarrow\infty} \frac{\#\{f\in\FF_D^*: r(E_f)=0\}}{\#\FF_D^*} = \lim_{D\rightarrow\infty} \frac{\#\{f\in\FF_D^*: r(E_f)=1\}}{\#\FF_D^*} = \frac{1}{2},
\end{equation*}
whereas
\begin{equation*}
 \lim_{D\rightarrow\infty} \frac{\#\{f\in\FF_D^*: r(E_f)\geq 2\}}{\#\FF_D^*} = 0.
\end{equation*}
\end{conjecture}
As the data on average rank above suggests, the distribution of analytic ranks for our four families is close to that predicted by the density conjecture with a non-trivial number of twists with rank greater than or equal to two.\footnote{The largest rank discovered was a rank 5 curve, a twist of $X_{222}/\F_5(t)$ by $f=t^7 +2t^6+ t^5 +4t^4+ 4t^3+ t^2+ 2t+1$.} 
We present the relevant data in the table below, where we have removed the dependence of the distribution on the degree of the twisting polynomials and instead consider all the $L$-functions we were able to compute given some $q$. For the dependence of the rank distribution on the degree, see the tables in the appendix.

\begin{table}[t]
\renewcommand{\arraystretch}{1.25}
\begin{center}
\begin{tabular}{c|c@{\hspace{3pt}}c@{\hspace{3pt}}c@{\hspace{3pt}}c|c@{\hspace{3pt}}c@{\hspace{3pt}}c@{\hspace{3pt}}c|c@{\hspace{3pt}}c@{\hspace{3pt}}c@{\hspace{3pt}}c|c@{\hspace{3pt}}c@{\hspace{3pt}}c@{\hspace{3pt}}c}\hline
$q$  & 5  &   &  &   &  
      7 &   & &   &  
     11 &  & &   & 
     13 &  & &   \\  
Rank & 0 & 1 & 2 & $\geq 3$ & 0 & 1 & 2 & $\geq 3$ & 0 & 1 & 2 & $\geq 3$ & 0 & 1 & 2 & $\geq 3$  \\ \noalign{\hrule height 2pt}
$X_{222}$ & .461 & .498 & .039 & .002 
          & .447 & .499 & .053 & .002 
          & .457 & .498 & .043 & .002 
          & .462 & .498 & .038 & .002 \\
$X_{211}$ & .483 & .500 & .018 & .000
          & .481 & .500 & .019 & .001 
          & .488 & .500 & .012 & .000
          & .481 & .500 & .019 & .000 \\
$X_{321}$ & .473 & .499 & .027 & .001 
          & .468 & .497 & .031 & .003 
          & .474 & .500 & .026 & .000
          & .474 & .499 & .026 & .000 \\
$X_{431}$ & .477 & .499 & .023 & .001
          & .464 & .499 & .036 & .001
          & .471 & .500 & .029 & .000
          & .474 & .499 & .026 & .001 \\ \hline
\end{tabular}

\vspace{.5cm}

\begin{tabular}{c|c@{\hspace{3pt}}c@{\hspace{3pt}}c@{\hspace{3pt}}c|c@{\hspace{3pt}}c@{\hspace{3pt}}c@{\hspace{3pt}}c|c@{\hspace{3pt}}c@{\hspace{3pt}}c@{\hspace{3pt}}c|c@{\hspace{3pt}}c@{\hspace{3pt}}c@{\hspace{3pt}}c}\hline
$q$  & 17  &   & &   &  
     19 &   & &   &  
     23 &   & &   & 
     29 &   & &   \\  
Rank & 0 & 1 & 2 & $\geq 3$ & 0 & 1 & 2 & $\geq 3$ & 0 & 1 & 2 & $\geq 3$ & 0 & 1 & 2 & $\geq 3$  \\ \noalign{\hrule height 2pt}
$X_{222}$ & .450 & .498 & .059 & .002 
          & .452 & .498 & .048 & .002 
          & .458 & .498 & .042 & .002
          & .463 & .499 & .036 & .001 \\
$X_{211}$ & .485 & .500 & .014 & .000
          & .484 & .500 & .016 & .000 
          & .488 & .500 & .012 & .000
          & .490 & .500 & .010 & .000 \\
$X_{321}$ & .471 & .500 & .029 & .000
          & .473 & .500 & .027 & .001
          & .475 & .500 & .025 & .000
          & .479 & .500 & .021 & .000 \\
$X_{431}$ & .463 & .499 & .037 & .000
          & .463 & .499 & .036 & .001
          & .470 & .499 & .030 & .001
          & .474 & .499 & .026 & .001 \\ \hline
\end{tabular}\caption{Rank Distributions for All Curves over All $d$ and $q\in\QQ$}
\end{center}
\end{table}
\renewcommand{\arraystretch}{1.0}


\section{Conclusion}

The remarkable property that the $L$-function of a non-isotrivial elliptic curve over a function field is a polynomial yields an effective algorithm to determine its coefficients by computing the number of points on a finite number of fibers. These fibers precisely correspond to the Euler factors that determine $L(E/K,T)$, and by realizing a given curve as a quadratic twist or pullback of another curve, the number of Euler factors that need to be computed can be minimized. In particular, the versal elliptic curve provides a (non-canonical) choice for an elliptic curve from which one can pullback and twist to recover any given elliptic curve, allowing for the efficient computation of the given curve's $L$-function provided sufficiently many Euler factors have been precomputed. These algorithms have been developed into \ellff, a software library for the open-source mathematical software system Sage, allowing anyone to quickly compute such $L$-functions. 

Experimentally, we computed the $L$-functions of four different families of quadratic twists in order to examine their analytic ranks for numerical evidence pertaining to Goldfeld's conjecture. Using an elementary argument, we know the asymptotic average rank over our family of quadratic twists is at least $1/2$ as the degree of the twisting polynomial becomes arbitrarily large. Unlike the situation in number fields, the case of function fields provides strong evidence, especially for the largest data sets, that this average is indeed $1/2$, thus supporting the validity of Goldfeld's conjecture in the function field case. Moreover, the experimental data also suggests that the analytic ranks are distributed closely with the density conjecture's prediction. Nonetheless the presence of a non-trivial amount of curves of rank at least 2 in even the largest data sets may suggest the convergence to this distribution is rather slow.

This work is part of a small but growing body of computational number theory directly focused on function fields. Historically, computational number theorists have primarily worked over number fields, in particular $\Q$. This (understandable) bias has produced a dearth of algorithms and data for the function field setting, despite the fact that many of the ideas from number fields can be formulated more generally for any global field. There is much work left to be done -- both theoretical and computational -- for the case of function fields, but we believe the example of $L$-functions of elliptic curves indicates that the effort is worthwhile and yields interesting mathematics.

\newpage

\section*{Appendix: Analytic Rank Distribution Tables}
The following tables give the distribution of analytic ranks, making explicit their dependence on the degree $d$ of twists considered in $\FF_d$.
\begin{table}[!h]
\renewcommand{\arraystretch}{1.25}
\begin{center}\tiny{
\begin{tabular}{cccccccc}\hline
& Rank & 0 & 1 & 2 & 3 & 4 & 5 \\ \noalign{\hrule height 2pt}
$q=5$ & $\FF_1$ & 0.333333 & 0.333333 & 0.333333 \\
& $\FF_2$ & 0.384615 & 0.615385 \\
& $\FF_3$ & 0.422535 & 0.521127 & 0.056338 \\
& $\FF_4$ & 0.417391 & 0.510145 & 0.072464 \\
& $\FF_5$ & 0.420357 & 0.496262 & 0.077631 & 0.005750 \\
& $\FF_6$ & 0.444163 & 0.497292 & 0.054973 & 0.003342 & 0.000230 \\
& $\FF_7$ & 0.442509 & 0.497017 & 0.057157 & 0.003087 & 0.000207 & 0.000023 \\
& $\FF_8$ & 0.459645 & 0.498025 & 0.040256 & 0.002009 & 0.000065  \\ 
& $\FF_9$ & 0.461748 & 0.498378 & 0.038150 & 0.001628 & 0.000095  \\ \cline{2-8} 
& All & 0.460615 & 0.498274 & 0.039264 & 0.001751 & 0.000094 & 0.000001 \\ \noalign{\hrule height 2pt}
$q=7$ & $\FF_1$ & 0.600000 & 0.400000 \\
& $\FF_2$ & 0.354839 & 0.580645 & 0.064516 \\
& $\FF_3$ & 0.418502 & 0.502203 & 0.070485 & 0.008811 \\
& $\FF_4$ & 0.419580 & 0.503497 & 0.076923 \\
& $\FF_5$ & 0.438684 & 0.497326 & 0.060818 & 0.003172 \\
& $\FF_6$ & 0.448610 & 0.498685 & 0.051280 & 0.001425 \\ \cline{2-8}  
& All & 0.446787 & 0.498634 & 0.052947 & 0.001632 \\ \noalign{\hrule height 2pt}
$q=11$ & $\FF_1$ & 0.444444 & 0.444444 & 0.111111 \\
& $\FF_2$ & 0.384615 & 0.549451 & 0.065934 \\
& $\FF_3$ & 0.442591 & 0.501472 & 0.052993 & 0.002944 \\
& $\FF_4$ & 0.433056 & 0.497809 & 0.066005 & 0.003041 & 0.000089 \\
& $\FF_5$ & 0.459054 & 0.498232 & 0.040828 & 0.001845 & 0.000041 \\ \cline{2-8} 
& All & 0.456731 & 0.498252 & 0.043021 & 0.001951 & 0.000044 \\ \noalign{\hrule height 2pt}
$q=13$ & $\FF_1$ & 0.454545 & 0.454545 & 0.090909 \\
& $\FF_2$ & 0.398496 & 0.533835 & 0.060150 & 0.007519 \\
& $\FF_3$ & 0.439749 & 0.499143 & 0.057110 & 0.003998 \\
& $\FF_4$ & 0.442430 & 0.497998 & 0.057064 & 0.002508 \\
& $\FF_5$ & 0.463263 & 0.498073 & 0.036630 & 0.001966 & 0.000068 \\ \cline{2-8} 
& All & 0.461629 & 0.498087 & 0.038204 & 0.002018 & 0.000062 \\ \noalign{\hrule height 2pt}
$q=17$ & $\FF_1$ & 0.400000 & 0.466667 & 0.133333 \\
& $\FF_2$ & 0.406639 & 0.526971 & 0.062241 & 0.004149 \\
& $\FF_3$ & 0.452144 & 0.495517 & 0.046038 & 0.006300 \\
& $\FF_4$ & 0.449574 & 0.498210 & 0.050162 & 0.002011 & 0.000043 \\ \cline{2-8} 
& All & 0.449568 & 0.498148 & 0.049989 & 0.002255 & 0.000040 \\ \noalign{\hrule height 2pt}
$q=19$ & $\FF_1$ & 0.470588 & 0.470588 & 0.058824 \\
& $\FF_2$ & 0.400651 & 0.521173 & 0.071661 & 0.006515 \\
& $\FF_3$ & 0.448440 & 0.500767 & 0.050111 & 0.000682 \\
& $\FF_4$ & 0.452392 & 0.498236 & 0.047451 & 0.001921 \\ \cline{2-8} 
& All & 0.452063 & 0.498419 & 0.047648 & 0.001871 \\ \noalign{\hrule height 2pt}
$q=23$ & $\FF_1$ & 0.476190 & 0.476190 & 0.047619 \\
& $\FF_2$ & 0.406048 & 0.522678 & 0.071274 \\
& $\FF_3$ & 0.458984 & 0.499111 & 0.039940 & 0.001871 & 0.000094 \\
& $\FF_4$ & 0.457616 & 0.498264 & 0.042271 & 0.001824 & 0.000023 \\ \cline{2-8} 
& All &  0.457577 & 0.498353 & 0.042221 & 0.001823 & 0.000027 \\ \noalign{\hrule height 2pt}
$q=29$ & $\FF_1$ & 0.407407 & 0.481481 & 0.111111 \\
& $\FF_2$ & 0.418758 & 0.515192 & 0.063408 & 0.002642 \\
& $\FF_3$ & 0.465170 & 0.498705 & 0.034216 & 0.001908 \\
& $\FF_4$ & 0.463392 & 0.498916 & 0.036535 & 0.001127 & 0.000030\\ \cline{2-8} 
& All & 0.463398 & 0.498927 & 0.036492 & 0.001154 & 0.000029 \\ \noalign{\hrule height 2pt}
\end{tabular}}\caption{Rank Distribution of $X_{222}/\F_q(t)$}\label{rkdist-0}
\end{center}
\end{table}

\begin{table}[htb]
\renewcommand{\arraystretch}{1.25}
\begin{center}\tiny{
\begin{tabular}{ccccccc}\hline
  & Rank & 0 & 1 & 2 & 3 & 4 \\ \noalign{\hrule height 2pt}
$q=5$ & $\FF_1$ & 0.666667 & 0.333333 \\
& $\FF_2$ & 0.307692 & 0.615385 & 0.076923 \\
& $\FF_3$ & 0.450704 & 0.521127 & 0.028169 \\
& $\FF_4$ & 0.405797 & 0.510145 & 0.084058 \\
& $\FF_5$ & 0.465210 & 0.502013 & 0.032777 \\
& $\FF_6$ & 0.427337 & 0.499942 & 0.072030 & 0.000691 \\
& $\FF_7$ & 0.483125 & 0.499758 & 0.016748 & 0.000369 \\
& $\FF_8$ & 0.456834 & 0.499611 & 0.043132 & 0.000424 \\
& $\FF_9$ & 0.488469 & 0.499751 & 0.011520 & 0.000256 & 0.000005\\ \cline{2-7} 
& All &  0.482791 & 0.499737 & 0.017179 & 0.000289 & 0.000004 \\ \noalign{\hrule height 2pt} 
$q=7$ & $\FF_1$ & 0.600000 & 0.400000 \\
& $\FF_2$ & 0.419355 & 0.580645 \\
& $\FF_3$ & 0.440529 & 0.511013 & 0.048458 \\
& $\FF_4$ & 0.471710 & 0.502225 & 0.024793 & 0.001271 \\
& $\FF_5$ & 0.467597 & 0.499864 & 0.031904 & 0.000634 \\
& $\FF_6$ & 0.482935 & 0.499566 & 0.016955 & 0.000544 \\ \cline{2-7} 
& All &  0.480738 & 0.499700 & 0.018996 & 0.000566 \\ \noalign{\hrule height 2pt} 
$q=11$ & $\FF_1$ & 0.555556 & 0.444444 \\
& $\FF_2$ & 0.450549 & 0.549451 \\
& $\FF_3$ & 0.481845 & 0.504416 & 0.013739 \\
& $\FF_4$ & 0.479116 & 0.499419 & 0.020034 & 0.001431 \\
& $\FF_5$ & 0.488324 & 0.499963 & 0.011599 & 0.000114 \\ \cline{2-7} 
& All &  0.487493 & 0.499982 & 0.012303 & 0.000222 \\ \noalign{\hrule height 2pt}
$q=13$ & $\FF_1$ & 0.454545 & 0.454545 & 0.090909 \\
& $\FF_2$ & 0.458647 & 0.541353 \\
& $\FF_3$ & 0.457453 & 0.503141 & 0.039406 \\
& $\FF_4$ & 0.484667 & 0.500506 & 0.014827 \\
& $\FF_5$ & 0.480937 & 0.499809 & 0.019024 & 0.000230 \\ \cline{2-7} 
& All &  0.481063 & 0.499892 & 0.018832 & 0.000212 \\ \noalign{\hrule height 2pt}
$q=17$ & $\FF_1$ & 0.533333 & 0.466667 \\
& $\FF_2$ & 0.468880 & 0.531120 \\
& $\FF_3$ & 0.480979 & 0.501817 & 0.017204 \\
& $\FF_4$ & 0.485716 & 0.500221 & 0.014063 \\ \cline{2-7} 
& All & 0.485409 & 0.500403 & 0.014189 \\ \noalign{\hrule height 2pt}
$q=19$ & $\FF_1$ & 0.470588 & 0.470588 & 0.058824 \\
& $\FF_2$ & 0.452769 & 0.527687 & 0.019544 \\
& $\FF_3$ & 0.468723 & 0.501108 & 0.029828 & 0.000341 \\
& $\FF_4$ & 0.484918 & 0.500157 & 0.014925 \\ \cline{2-7} 
& All & 0.484024 & 0.500272 & 0.015687 & 0.000017 \\ \noalign{\hrule height 2pt}
$q=23$ & $\FF_1$ & 0.476190 & 0.476190 & 0.047619 \\
& $\FF_2$ & 0.457883 & 0.522678 & 0.019438 \\
& $\FF_3$ & 0.487513 & 0.500889 & 0.011505 & 0.000094 \\
& $\FF_4$ & 0.487659 & 0.500089 & 0.012252 \\ \cline{2-7} 
& All &  0.487590 & 0.500171 & 0.012235 & 0.000004 \\ \noalign{\hrule height 2pt}
$q=29$ & $\FF_1$ & 0.518519 & 0.481481 \\
& $\FF_2$ & 0.467635 & 0.517834 & 0.014531 \\
& $\FF_3$ & 0.488345 & 0.500523 & 0.011042 & 0.000091 \\
& $\FF_4$ & 0.489674 & 0.499968 & 0.010283 & 0.000075 \\ \cline{2-7} 
& All & 0.489605 & 0.500006 & 0.010313 & 0.000076 \\ \noalign{\hrule height 2pt}
\end{tabular}}\caption{Rank Distribution of $X_{211}/\F_q(t)$}\label{rkdist-1}
\end{center}
\end{table}

\begin{table}[htb]
\renewcommand{\arraystretch}{1.25}
\begin{center}\tiny{
\begin{tabular}{ccccccc}\hline
& Rank & 0 & 1 & 2 & 3 & 4 \\ \noalign{\hrule height 2pt}
$q=5$ & $\FF_1$ & 0.666667 & 0.333333 \\  
& $\FF_2$ & 0.384615 & 0.615385 \\  
& $\FF_3$ & 0.380282 & 0.521127 & 0.098592 \\  
& $\FF_4$ & 0.446377 & 0.504348 & 0.043478 & 0.005797 \\  
& $\FF_5$ & 0.444508 & 0.498562 & 0.053479 & 0.003450 \\  
& $\FF_6$ & 0.452345 & 0.498790 & 0.047021 & 0.001844 \\  
& $\FF_7$ & 0.460525 & 0.497892 & 0.039302 & 0.002235 & 0.000046 \\  
& $\FF_8$ & 0.469898 & 0.498468 & 0.030059 & 0.001567 & 0.000009 \\  
& $\FF_9$ & 0.474192 & 0.499060 & 0.025765 & 0.000947 & 0.000036 \\ \cline{2-7}   
& All & 0.472877 & 0.498928 & 0.027065 & 0.001098 & 0.000032 \\ \noalign{\hrule height 2pt} 
$q=7$ & $\FF_1$ & 0.600000 & 0.400000 \\   
& $\FF_2$ & 0.419355 & 0.580645 \\   
& $\FF_3$ & 0.458150 & 0.511013 & 0.030837 \\  
& $\FF_4$ & 0.448188 & 0.495868 & 0.048315 & 0.007629 \\  
& $\FF_5$ & 0.468504 & 0.499592 & 0.030998 & 0.000906 \\  
& $\FF_6$ & 0.0.468881 & 0.496872 & 0.030970 & 0.003238 & 0.000039 \\ \cline{2-7}   
& All & 0.468436 & 0.497247 & 0.031264 & 0.003020 & 0.000033 \\ \noalign{\hrule height 2pt} 
$q=11$ & $\FF_1$ & 0.555556 & 0.444444 \\   
& $\FF_2$ & 0.450549 & 0.549451 \\   
& $\FF_3$ & 0.457311 & 0.504416 & 0.038273 \\  
& $\FF_4$ & 0.465254 & 0.500850 & 0.033897 \\  
& $\FF_5$ & 0.475002 & 0.499614 & 0.024921 & 0.000463 \\ \cline{2-7}   
& All & 0.474052 & 0.499782 & 0.025745 & 0.000421 \\ \noalign{\hrule height 2pt}
$q=13$ & $\FF_1$ & 0.454545 & 0.454545 & 0.090909 \\   s
& $\FF_2$ & 0.421053 & 0.541353 & 0.037594 \\  
& $\FF_3$ & 0.448886 & 0.500857 & 0.047973 & 0.002284 \\  
& $\FF_4$ & 0.464297 & 0.500506 & 0.035197 \\  
& $\FF_5$ & 0.0.474582 & 0.499125 & 0.025379 & 0.000914 \\ \cline{2-7}   
& All & 0.473689 & 0.499249 & 0.026207 & 0.000856 \\ \noalign{\hrule height 2pt}
$q=17$ & $\FF_1$ & 0.466667 & 0.466667 & 0.066667 \\  
& $\FF_2$ & 0.439834 & 0.531120 & 0.029046 \\  
& $\FF_3$ & 0.459414 & 0.500363 & 0.038769 & 0.001454 \\  
& $\FF_4$ & 0.471696 & 0.499594 & 0.028083 & 0.000628\\ \cline{2-7}   
& All & 0.470911 & 0.499732 & 0.028686 & 0.000671 \\ \noalign{\hrule height 2pt}
$q=19$ & $\FF_1$ & 0.529412 & 0.470588 \\  
& $\FF_2$ & 0.442997 & 0.527687 & 0.029316 \\  
& $\FF_3$ & 0.467701 & 0.501108 & 0.030851 & 0.000341 \\  
& $\FF_4$ & 0.473385 & 0.499574 & 0.026458 & 0.000583 \\ \cline{2-7}   
& All & 0.473030 & 0.499719 & 0.026681 & 0.000570 \\ \noalign{\hrule height 2pt}
$q=23$ & $\FF_1$ & 0.476190 & 0.476190 & 0.047619 \\   
& $\FF_2$ & 0.444924 & 0.522678 & 0.032397 \\  
& $\FF_3$ & 0.474418 & 0.500889 & 0.024600 & 0.000094 \\  
& $\FF_4$ & 0.474963 & 0.499584 & 0.024939 & 0.000505 & 0.000009 \\ \cline{2-7}   
& All & 0.474875 & 0.499691 & 0.024940 & 0.000485 & 0.000009 \\ \noalign{\hrule height 2pt}
$q=29$ & $\FF_1$ & 0.518519 & 0.481481 \\
& $\FF_2$ & 0.446499 & 0.517834 & 0.035667 \\  
& $\FF_3$ & 0.471986 & 0.499659 & 0.027400 & 0.000954 \\  
& $\FF_4$ & 0.478870 & 0.499689 & 0.021082 & 0.000354 & 0.000005 \\ \cline{2-7}   
& All & 0.478605 & 0.499708 & 0.021308 & 0.000374 & 0.000005 \\ \noalign{\hrule height 2pt}
\end{tabular}}\caption{Rank Distribution of $X_{321}/\F_q(t)$}\label{rkdist-2}
\end{center}
\end{table}

\begin{table}[htb]
\renewcommand{\arraystretch}{1.25}
\begin{center}\tiny{
\begin{tabular}{ccccccc}\hline
  & Rank & 0 & 1 & 2 & 3 & 4 \\ \noalign{\hrule height 2pt}
$q=5$ & $\FF_1$ & 0.666667 & 0.333333 \\ 
& $\FF_2$ & 0.384615 & 0.615385 \\
& $\FF_3$ & 0.422535 & 0.521127 & 0.056338 \\
& $\FF_4$ & 0.449275 & 0.498551 & 0.040580 & 0.011594 \\
& $\FF_5$ & 0.456009 & 0.502013 & 0.041978 \\
& $\FF_6$ & 0.452345 & 0.496254 & 0.047021 & 0.004379 \\
& $\FF_7$ & 0.466169 & 0.499551 & 0.033704 & 0.000576 \\
& $\FF_8$ & 0.469543 & 0.497307 & 0.030377 & 0.002728 & 0.000046 \\
& $\FF_9$ & 0.478878 & 0.499746 & 0.021099 & 0.000261 & 0.000016 \\ \cline{2-7} 
& All &  0.476768 & 0.499332 & 0.023186 & 0.000695 & 0.000020 \\ \noalign{\hrule height 2pt} 
$q=7$ & $\FF_1$ & 0.400000 & 0.400000 & 0.200000 \\
& $\FF_2$ & 0.419355 & 0.580645 \\
& $\FF_3$ & 0.418502 & 0.511013 & 0.070485 \\
& $\FF_4$ & 0.442467 & 0.503497 & 0.054037 \\
& $\FF_5$ & 0.454274 & 0.497507 & 0.045228 & 0.002991 \\
& $\FF_6$ & 0.465992 & 0.499268 & 0.033768 & 0.000842 & 0.000130 \\ \cline{2-7} 
& All & 0.464007 & 0.499178 & 0.035616 & 0.001088 & 0.000111 \\ \noalign{\hrule height 2pt} 
$q=11$ & $\FF_1$ & 0.555556 & 0.444444 \\
& $\FF_2$ & 0.395604 & 0.549451 & 0.054945 \\
& $\FF_3$ & 0.448479 & 0.504416 & 0.047105 \\
& $\FF_4$ & 0.453537 & 0.500850 & 0.045613 \\
& $\FF_5$ & 0.473027 & 0.499858 & 0.026896 & 0.000219 \\ \cline{2-7} 
& All & 0.471185 & 0.500004 & 0.028612 & 0.000200 \\ \noalign{\hrule height 2pt}
$q=13$ & $\FF_1$ & 0.454545 & 0.454545 & 0.090909 \\
& $\FF_2$ & 0.421053 & 0.541353 & 0.037594 \\ 
& $\FF_3$ & 0.439178 & 0.495717 & 0.057681 & 0.007424 \\
& $\FF_4$ & 0.457609 & 0.499186 & 0.041885 & 0.001320 \\
& $\FF_5$ & 0.475320 & 0.498882 & 0.024611 & 0.001157 & 0.000030 \\ \cline{2-7} 
& All & 0.473842 & 0.498902 & 0.026026 & 0.001203 & 0.000028 \\ \noalign{\hrule height 2pt}
$q=17$ & $\FF_1$ & 0.533333 & 0.466667 \\ 
& $\FF_2$ & 0.423237 & 0.531120 & 0.045643 \\
& $\FF_3$ & 0.460383 & 0.501333 & 0.037800 & 0.000485 \\
& $\FF_4$ & 0.462981 & 0.499280 & 0.036798 & 0.000941 \\ \cline{2-7} 
& All & 0.462723 & 0.499490 & 0.036874 & 0.000913 \\ \noalign{\hrule height 2pt}
$q=19$ & $\FF_1$ & 0.470588 & 0.470588 & 0.058824 \\
& $\FF_2$ & 0.433225 & 0.524430 & 0.039088 & 0.003257 \\
& $\FF_3$ & 0.456792 & 0.497529 & 0.041759 & 0.003920 \\
& $\FF_4$ & 0.463782 & 0.499260 & 0.036043 & 0.000897 & 0.000018 \\ \cline{2-7} 
& All & 0.463354 & 0.499235 & 0.036340 & 0.001054 & 0.000017 \\ \noalign{\hrule height 2pt}
$q=23$ & $\FF_1$ & 0.476190 & 0.476190 & 0.047619 \\
& $\FF_2$ & 0.429806 & 0.522678 & 0.047516 \\
& $\FF_3$ & 0.467590 & 0.500514 & 0.031428 & 0.000468 \\
& $\FF_4$ & 0.469808 & 0.499289 & 0.030094 & 0.000800 & 0.000009 \\ \cline{2-7} 
& All & 0.469620 & 0.499393 & 0.030195 & 0.000782 & 0.000009 \\ \noalign{\hrule height 2pt}
$q=29$ & $\FF_1$ & 0.481481 & 0.481481 & 0.037037 \\
& $\FF_2$ & 0.433289 & 0.517834 & 0.048877 \\
& $\FF_3$ & 0.472713 & 0.500250 & 0.026673 & 0.000364 \\
& $\FF_4$ & 0.474451 & 0.499316 & 0.025503 & 0.000727 & 0.000003 \\ \cline{2-7} 
& All & 0.474346 & 0.499368 & 0.025569 & 0.000714 & 0.000003 \\ \noalign{\hrule height 2pt}
\end{tabular}}\caption{Rank Distribution of $X_{431}/\F_q(t)$}\label{rkdist-3}
\end{center}
\end{table}

\end{document}